\documentclass{IEEEtran}
\usepackage{cite}
\usepackage{amsmath,amssymb,amsfonts}
\usepackage{multirow}
\usepackage{algorithmic}
\usepackage{amsfonts,cite,times,url,hyperref,bm}
\usepackage{graphicx}
\usepackage{textcomp}
\usepackage{color}
\usepackage{circuitikz}

\newtheorem{lemma}{Lemma}

\newtheorem{definition}{Definition}

\newtheorem{approximation}{Approximation}

\def\BibTeX{{\rm B\kern-.05em{\sc i\kern-.025em b}\kern-.08em
    T\kern-.1667em\lower.7ex\hbox{E}\kern-.125emX}}
\begin{document}
\title{Geometry of Distance Protection}
\author{Joshua A. Taylor, \IEEEmembership{Senior Member, IEEE}, Alejandro D. Dom\'inguez-Garc\'ia, \IEEEmembership{Fellow, IEEE}
\thanks{This work was supported in part by the National Science Foundation under Grant 2411925.}
\thanks{Josh A. Taylor is with the Department of Electrical and Computer Engineering,
       New Jersey Institute of Technology, Newark, NJ, USA. E-mail:
        {\tt\small jat94@njit.edu}}
\thanks{Alejandro D. Dom\'inguez-Garc\'ia is with the Department of Electrical and Computer Engineering, University of Illinois at 
	Urbana-Champaign, Urbana, IL, USA. E-mail:
        {\tt\small aledan@illinois.edu}}
}

\maketitle

\begin{abstract}
Distance relays detect faults on transmission lines. They face uncertainty from the fault's location and resistance, as well as the current from the line's remote terminal. In this paper, we aggregate this uncertainty with the Minkowski sum. This allows us to explicitly model the power grid surrounding the relay's line, and in turn accommodate any mix of synchronous machines and inverter-based resources. To make the relay's task easier, inverters can inject perturbations, or auxiliary signals, such as negative-sequence current. We use Farkas' lemma to construct an optimization for designing inverter auxiliary signals.
\end{abstract}

\begin{IEEEkeywords}
Distance protection, Minkowski sum, zonotope, auxiliary signal, Farkas' lemma, inverter-based resource.
\end{IEEEkeywords}

\section{Introduction}\label{sec:intro}

Faults in power systems are, for the most part, unintentional short circuits. Common causes include lightning, trees, and animals. The distance or phasor-based relay is an established technology for detecting faults on transmission lines. It functions by measuring the local voltage and current and computing an apparent impedance to the fault. From this it determines if there is a fault on its line, in which case it should trip a circuit breaker. We refer the reader to~\cite{ziegler2011numerical} and~\cite{horowitz2022power} for background on distance protection.

In addition to the fault's location, a distance relay faces uncertainty from the fault's resistance and any current flowing into the line's remote terminal, which it cannot measure. To accommodate this uncertainty, the relay responds when the apparent impedance is in a certain region of the complex plane, which is known as a characteristic. A sophisticated type of characteristic is the quadrilateral, which is intuitively formed by adding the uncertainties of the fault's location and resistance and the remote current. The remote current, which depends on the power grid surrounding the relay's line, is incorporated via a physically motivated, ad hoc adjustment. Our first main contribution is the formalization of this accounting in terms of the Minkowski sum~\cite{ziegler1995lectures}.

The characteristics of distance relays---particularly the remote-current adjustment---are based on the assumption that the power sources are synchronous generators (SGs). There is now ample evidence that distance relays can misoperate in grids with significant amounts of inverter-based resources (IBRs)~\cite{kasztenny2022distance,baeckeland2022distance,protection2024gap}. By casting the relay characteristics as Minkowski sums, we are able to explicitly model the network and sources, which allows us to accommodate any combination of SGs and IBRs.

Another remedy for relay misoperation is for the IBRs to inject perturbations such as harmonics~\cite{saleh2020protection} and negative-sequence current~\cite{banaiemoqadam2019control}. This is now a part of some grid codes; e.g., IEEE Standard 2800~\cite{IEEE28002022} stipulates that each IBR inject negative-sequence current that leads its negative-sequence terminal voltage by $90^\circ$. Perturbations that aid fault and failure detection are also known as auxiliary signals~\cite{campbell2015auxiliary}. Our second main contribution is an optimization for designing generic auxiliary signals for IBRs. We construct the optimization using Farkas' lemma~\cite{Schrijver1998LPIP} and by building on our prior work~\cite{taylor2024fault}.

In Section~\ref{sec:modeling}, we present basic models of the power grid, IBR auxiliary signals, and the fault loops of the eleven three-phase faults (all of which are captured in two analyses). We treat the power grid as a lumped circuit in which SGs are voltage sources and IBRs current sources. While this last assumption has some precedent~\cite{banaiemoqadam2019control}, it is crude, and could in principle be replaced by another approximation so long as it is linear or convex.

In Section~\ref{sec:uncertainty}, we model uncertainty from the fault's location and resistance and the remote terminal current. We also qualitatively discuss intermediate infeeds~\cite{ziegler2011numerical}. We distinguish between \textit{pre-test} uncertainty, prior to measuring the voltage and current, and \textit{post-test} uncertainty. The auxiliary signal must be designed offline, pre-test. Post-test, the relay has observed the local voltage and current and must determine if there is a fault on its line. Both pre- and post-test, we aggregate the uncertainty via the Minkowski sum.

The post-test uncertainty sets contain all possible apparent impedances the relay can see. As such, we regard them as a new type of relay characteristic. Absent intermediate and remote infeeds, the post-test uncertainty sets are equivalent to quadrilateral characteristics. It is, however, through these sources of uncertainty that IBRs create problems for quadrilateral characteristics~\cite{banaiemoqadam2019control,baeckeland2022distance}, and which we precisely account for with the Minkowski sum.

To use the post-test uncertainty sets as characteristics, we must project them onto the space of impedances, i.e., the complex plane. In Section~\ref{sec:projection}, we show that if we model the SG and IBR uncertainties as zonotopes~\cite{ziegler1995lectures}, then much of this structure carries through, and the most intensive computation is the gift wrapping algorithm~\cite{jarvis1973identification}. We also give a zonogon approximation with analytical form in Section~\ref{sec:zonogon}.

The pre-test uncertainty sets may be left in high dimensions, but have bilinear products of uncertain variables. In Section~\ref{sec:convex}, we derive several convex approximations. We are ultimately concerned with the intersections of pairs of pre-test uncertainty sets, e.g., normal operation and phase a-to-phase b faults, or phase a-to-phase b and phase a-to-ground faults. If the intersection between a pair of pre-test uncertainty sets is empty, then there is no possible observation the could be consistent with the two corresponding scenarios. If true for all pairs of scenarios, the relay will, in theory, never face an ambiguous decision.

If such an intersection is not already empty, it can be made to be empty via an auxiliary signal, which we otherwise want to be as small as possible. In Section~\ref{sec:auxiliary}, we use Farkas' lemma~\cite{Schrijver1998LPIP} to obtain dual systems that are feasible exactly when the corresponding intersections are empty. This enables us to optimize auxiliary signals subject to the constraint that the relay never faces an ambiguous decision. Our construction largely follows our earlier work~\cite{taylor2024fault}, which builds on the formulations of~\cite{campbell2015auxiliary}. The optimization is a bilinear program, for which we give an algorithm based on the Alternating Direction Method of Multipliers~\cite{boyd2011distributed} in Appendix~\ref{app:admm}.

We give examples of post-test uncertainty sets in Section~\ref{sec:example:postest} and auxiliary signal design in Section~\ref{sec:example:AS}. All computations were carried out in Python using open-source tools.

\section{Physical model}\label{sec:modeling}

\subsection{Network}
The voltage and current at bus $k$ are 
\[
v_{k}=\begin{bmatrix}
v_{k}^{\textrm{a}}\\
v_{k}^{\textrm{b}}\\
v_{k}^{\textrm{c}}
\end{bmatrix}\in\mathbb{C}^3
\quad\textrm{and}\quad
i_{k}=\begin{bmatrix}
i_{k}^{\textrm{a}}\\
i_{k}^{\textrm{b}}\\
i_{k}^{\textrm{c}}
\end{bmatrix}\in\mathbb{C}^3.
\]
Let $v\in\mathbb{C}^{3n}$ and $i\in\mathbb{C}^{3n}$ be the stacked three-phase voltage and current injection vectors for all $n$ buses. Let $\mathcal{S}$ and $\mathcal{C}$ be the subsets of buses with SGs and IBRs, respectively, and let subscripted $\mathcal{S}$ and $\mathcal{C}$ indicate the corresponding subvector, e.g., $v_{\mathcal{S}}$ is the subvector of voltages at buses with SGs.

Consider a symmetrical, fully transposed line in the network with positive-sequence series impedance $z$. A relay is at the line's local bus, $\textrm{L}$. The remote bus is denoted $\textrm{R}$. The relay measures the local bus voltages, $v_{\textrm{L}}$, and currents flowing into the line from the local bus, $i_{\textrm{L}}$. $v_{\textrm{R}}$ and $i_{\textrm{R}}$ are the analogous quantities at the remote side, where $i_{\textrm{R}}$ also flows into the line.

Suppose a fault occurs somewhere in the network. The fault's resistance is $m_rr_{\textrm{F}}$, where $m_r\in\left[0,1\right]$ is its normalized resistance and $r_{\textrm{F}}>0$. The impedance of the portion of the line between the relay and the fault is $m_zz$, where $m_z\in\mathbb{R}$ is the normalized distance. There is a fault on the line if $m_z\in[0,1]$. Here, the relay tries to determine if $m_z\in[\underbar{m}_z,1]$ for some $\underbar{m}_z\in[0,1)$. `Close-in' faults for which $m_z\in[0,\underbar{m}_z)$ are typically handled as exceptions, e.g., via dead zones~\cite{ziegler2011numerical}.

Let $\alpha=e^{j2\pi/3}$,
\[
A=\begin{bmatrix}
1&1&1\\
1&\alpha^2&\alpha\\
1&\alpha&\alpha^2
\end{bmatrix},\textrm{ and }
B=A^{-1}=\frac{1}{3}\begin{bmatrix}
1&1&1\\
1&\alpha&\alpha^2\\
1&\alpha^2&\alpha
\end{bmatrix}.
\]
The symmetrical components of the voltage at bus $k$ are
\[
\begin{bmatrix}
v_k^0\\
v_k^+\\
v_k^-
\end{bmatrix} = Bv_k.
\]
Symmetrical components of currents are similarly defined.

To extract individual phase information, we will sometimes use the canonical vectors
\begin{align*}
e^{\textrm{a}}&=\begin{bmatrix}1&0&0\end{bmatrix},\;e^{\textrm{b}}=\begin{bmatrix}0&1&0\end{bmatrix},\;e^{\textrm{c}}=\begin{bmatrix}0&0&1\end{bmatrix},
\end{align*}
and for zero-sequence components,
\[
e^0=\frac{1}{3}\begin{bmatrix}1&1&1\end{bmatrix}.
\]

\subsection{Auxiliary signals}

In this paper, the auxiliary signal consists of perturbations to the IBR currents. Let $\Delta$ be the vector of auxiliary signals. The IBR currents, $i_{\mathcal{C}}$, are affine functions of $\Delta$. We will write $i_{\mathcal{C}}(\Delta)$ to make the dependence explicit starting in Section~\ref{sec:protection}.

A good choice is $\Delta=i_{\mathcal{C}}^-$, negative-sequence current, so that for $k\in\mathcal{C}$,
\[
i_k(\Delta_k) = A\begin{bmatrix}0\\i_k^+\\\Delta_k
\end{bmatrix}.
\]
Some convenient features of negative-sequence current are that, unlike, e.g., harmonics, positive and negative sequence impedances are the same if the network is symmetrical, simplifying modeling; and negative-sequence currents are not attenuated by high impedances.


Other choices include harmonic currents (which entails modeling the system at harmonic frequencies), zero and positive-sequence currents, and per-phase current injections, all of which can be represented as affine functions of $\Delta$.

We might also include convex constraints on the auxiliary signal, e.g., that the current magnitudes be within the IBRs' limits:
\begin{align*}
\left\|i_k(\Delta_k) \right\|_2 \leq i_k^{\max},\quad k\in\mathcal{C}.
\end{align*}

\subsection{Fault loops}\label{sec:faultloops}

A relay must determine if there is a fault on its line and, if so, which type. It must distinguish between twelve scenarios: normal operation (N), and the eleven types of faults---line-to-ground (LG), line-to-line (LL), double line-to-ground (LLG), and the symmetrical faults, abc and abcg. Each scenario has a model, which we index with the elements of the set
\[
\mathbb{F} = \{
\underset{\textrm{LG}}{\underbrace{\textrm{ag},\textrm{bg},\textrm{cg}}},
\underset{\textrm{LL}}{\underbrace{\textrm{ab},\textrm{ac},\textrm{bc}}},
\underset{\textrm{LLG}}{\underbrace{\textrm{abg},\textrm{acg},\textrm{bcg}}},
\underset{\textrm{symmetrical}}{\underbrace{\textrm{abc},\textrm{abcg}}},\textrm{N}
\}.
\]
Let
\[
\mathbb{F}_2=\mathbb{F}\times\mathbb{F}\setminus\left\{(\eta,\eta)\;|\;\eta\in\mathbb{F}\right\},
\]
i.e., the Cartesian product of $\mathbb{F}$ with itself, excluding duplicates. We indicate the scenario with a superscript.

The relay infers which model in $\mathbb{F}$, and thus which scenario, is true from its measurements, $v_{\textrm{L}}$ and $i_{\textrm{L}}$. It relates these measurements to the fault by computing KVL around the fault loop, which is the path from the relay to the fault and back~\cite{ziegler2011numerical}. Because the LG and LL fault loops are present for all of the other types of faults, we can use their analyses for all fault types~\cite{kasztenny2008fundamentals}. We assume a `non-switched' relay that continuously samples the voltage and current in all three phases, $v_{\textrm{L}}$ and $i_{\textrm{L}}$~\cite{ziegler2011numerical}.

For each scenario $\eta\in\mathbb{F}$, from its measurements, the relay computes an apparent voltage drop, $v_{\textrm{A}}^{\eta}$, current, $i_{\textrm{A}}^{\eta}$, and impedance, $z_{\textrm{A}}^{\eta}=v_{\textrm{A}}^{\eta}/i_{\textrm{A}}^{\eta}$. For example, for an ab fault,
\begin{align*}
v_{\textrm{A}}^{\textrm{ab}} &= v_{\textrm{L}}^{\textrm{a}}-v_{\textrm{L}}^{\textrm{b}}\quad\textrm{and}\quad i_{\textrm{A}}^{\textrm{ab}} = i_{\textrm{L}}^{\textrm{a}}-i_{\textrm{L}}^{\textrm{b}}.
\end{align*}
For each $\eta\in\mathbb{F}$, we let $\psi^{\eta}$ denote the mapping from $v_{\textrm{L}}$ to $v_{\textrm{A}}^{\eta}$; e.g., $\psi^{\textrm{ab}}=[1,-1,0]$.

For all fault types, KVL has two parts: the voltage drops due to the transmission line, $v_{\textrm{T}}$, and due to the fault, $v_{\textrm{F}}$. Dividing through by $i_{\textrm{A}}$, we obtain the impedance due to the line, $z_{\textrm{T}}=m_zz$, and the apparent impedance due to the fault, $z_{\textrm{F}}$.
Then
\begin{subequations}
\label{eq:KVL}
\begin{align}
v_{\textrm{A}} &= v_{\textrm{T}} + v_{\textrm{F}}\label{eq:KVLv}\\
z_{\textrm{A}} &= z_{\textrm{T}}+ z_{\textrm{F}}.\label{eq:KVLz}
\end{align}
\end{subequations}
In all cases, $v_{\textrm{F}}=m_rr_{\textrm{F}}i_{\textrm{F}}$, where $i_{\textrm{F}}$ is the current through the fault. 

If the fault is bolted, i.e., $m_r=0$, then $v_{\textrm{F}}=z_{\textrm{F}}=0$, and the fault's location is simply $m_z=z_{\textrm{A}}/z$. If $m_r>0$, but there is no remote infeed, i.e., $i_{\textrm{R}}=0$, then $z_{\textrm{F}}=m_rr_{\textrm{F}}$. In this case, we can take the imaginary part of (\ref{eq:KVLz}) to obtain $m_z=x_{\textrm{A}}/x$~\cite{ziegler2011numerical}. The relay, however, cannot assume that $m_r=0$ or $i_{\textrm{R}}=0$, and so must allow for uncertainty in $v_{\textrm{F}}$ and $z_{\textrm{F}}$.

In this section, we specify (\ref{eq:KVL}) for normal operation and ag and ab faults, which we use as nominal instances of LG and LL faults. In Section~\ref{sec:uncertainty}, we compute corresponding uncertainty sets using the Minkowski sum.

\subsubsection{LG faults}\label{sec:lg}

Consider a fault from phase a to ground, as in Fig.~\ref{fig:ag}.
\begin{figure}[h]
\begin{center}
\begin{circuitikz}
	\draw (0,0) node[above]{$v_{\textrm{L}}^{\textrm{a}}$} to [open, *-*] (8,0) node[above]{$v_{\textrm{R}}^{\textrm{a}}$};
	\draw (0,0) to [generic=$m_zz$ , i>_=$i_{\textrm{L}}^{\textrm{a}}$] (4,0)
	to [generic=$(1-m_z)z$ ,  i_<=$i_{\textrm{R}}^{\textrm{a}}$] (8,0);
  	\draw (4,0) to [/tikz/circuitikz/bipoles/length=20pt,R=$m_rr_{\textrm{F}}$,  i_>=$i_{\textrm{F}}$] (4,-1.5)
	to (4,-1.5) node[ground]{};
	\draw (4,0) to [open, *-] (4,-1.5);
	\draw (0,1.5) node[above]{$v_{\textrm{L}}^{\textrm{b}}$} to [open, *-*, i_<=$i_{\textrm{R}}^{\textrm{b}}$] (8,1.5) node[above]{$v_{\textrm{R}}^{\textrm{b}}$};
	\draw (0,1.5) to [generic=$z$ , i>_=$i_{\textrm{L}}^{\textrm{b}}$] (8,1.5);
	\draw (0,3) node[above]{$v_{\textrm{L}}^{\textrm{c}}$} to [open, *-*, i_<=$i_{\textrm{R}}^{\textrm{c}}$] (8,3) node[above]{$v_{\textrm{R}}^{\textrm{c}}$};
	\draw (0,3) to [generic=$z$ , i>_=$i_{\textrm{L}}^{\textrm{c}}$] (8,3);
\end{circuitikz}
\end{center}
\caption{Circuit representation of an ag fault. The fault loop is from phase a of the local bus to ground.}
\label{fig:ag}
\end{figure}
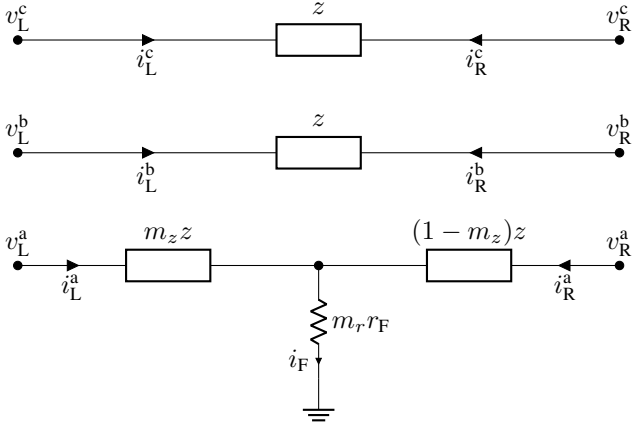
Let $z^0$ be the zero-sequence impedance of the line and $k = z^0/z - 1$ the zero-sequence compensation factor.\footnote{Line-to-ground fault currents typically have zero-sequence content for which we must use zero-sequence impedances~\cite{ziegler2011numerical}. This is not the case for line-to-line faults.} We set
\begin{align*}
v_{\textrm{A}}^{\textrm{ag}} &= v_{\textrm{L}}^{\textrm{a}}, \quad i_{\textrm{A}}^{\textrm{ag}} = i_{\textrm{L}}^{\textrm{a}}+ki_{\textrm{L}}^0,
\end{align*}
where $i_{\textrm{L}}^0=e^0i_{\textrm{L}}$. Here $\psi^{\textrm{ag}}=[1,0,0]$. KVL from the relay to ground gives
\begin{subequations}
\begin{align}
v_{\textrm{L}}^{\textrm{a}}&= z_{\textrm{A}}^{\textrm{ag}}\left(i_{\textrm{L}}^{\textrm{a}}+ki_{\textrm{L}}^0\right)\\
 &=  \underset{=v_{\textrm{T}}^{\textrm{ag}}}{\underbrace{ m_z z \left(i_{\textrm{L}}^{\textrm{a}}+ki_{\textrm{L}}^0\right)}} +  \underset{=v_{\textrm{F}}^{\textrm{ag}}}{\underbrace{m_r r_{\textrm{F}} \left(i_{\textrm{L}}^{\textrm{a}}+ i_{\textrm{R}}^{\textrm{a}}\right)}}.\label{eq:lgkvl}
\end{align}
Dividing through by $i_{\textrm{L}}^{\textrm{a}}+ki_{\textrm{L}}^0$, we have
\begin{align}
z_{\textrm{A}}^{\textrm{ag}} &=  \underset{=z_{\textrm{T}}}{\underbrace{m_z z}}  +  \underset{=z_{\textrm{F}}^{\textrm{ag}}}{\underbrace{m_r r_{\textrm{F}} \frac{i_{\textrm{L}}^{\textrm{a}}+ i_{\textrm{R}}^{\textrm{a}}}{i_{\textrm{L}}^{\textrm{a}}+ki_{\textrm{L}}^0}}}.\label{eq:lgkvlz}
\end{align}
\end{subequations}

\subsubsection{LL faults}\label{sec:ll}
Consider a fault between phases $\textrm{a}$ and $\textrm{b}$, as in Fig.~\ref{fig:ab}.
\begin{figure}[h]
\begin{center}
\begin{circuitikz}
	\draw (0,0) node[above]{$v_{\textrm{L}}^{\textrm{a}}$} to [open, *-*] (8,0) node[above]{$v_{\textrm{R}}^{\textrm{a}}$};
	\draw (0,0) to [generic=$m_zz$ , i>_=$i_{\textrm{L}}^{\textrm{a}}$] (4,0)
	to [generic=$(1-m_z)z$ ,  i_<=$i_{\textrm{R}}^{\textrm{a}}$] (8,0);
	\draw (0,1.5) node[above]{$v_{\textrm{L}}^{\textrm{b}}$} to [open, *-*] (8,1.5) node[above]{$v_{\textrm{R}}^{\textrm{b}}$};
	\draw (0,1.5) to [generic=$m_zz$ , i>_=$i_{\textrm{L}}^{\textrm{b}}$] (4,1.5)
	to [generic=$(1-m_z)z$ ,  i_<=$i_{\textrm{R}}^{\textrm{b}}$] (8,1.5);
	\draw (0,3) node[above]{$v_{\textrm{L}}^{\textrm{c}}$} to [open, *-*, i_<=$i_{\textrm{R}}^{\textrm{c}}$] (8,3) node[above]{$v_{\textrm{R}}^{\textrm{c}}$};
	\draw (0,3) to [generic=$z$ , i>_=$i_{\textrm{L}}^{\textrm{c}}$] (8,3);
\draw (4,0) to [/tikz/circuitikz/bipoles/length=20pt,R=$m_rr_{\textrm{F}}$,  i>_=$i_{\textrm{F}}$] (4,1.5)
	to (4,1.5);
	\draw (4,0) to [open, *-*] (4,1.5);
\end{circuitikz}
\end{center}
\caption{Circuit representation of an ab fault. The fault loop is from phase a to phase b of the local bus.}
\label{fig:ab}
\end{figure}
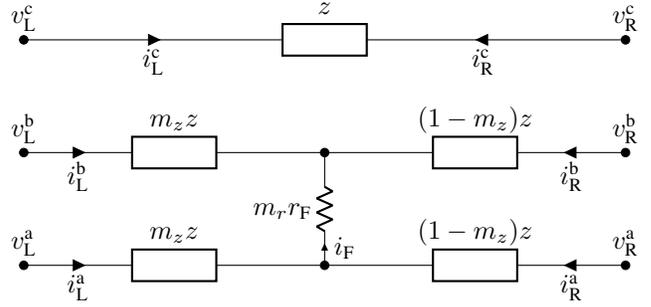
We set
\begin{align*}
v_{\textrm{A}}^{\textrm{ab}} &= v_{\textrm{L}}^{\textrm{a}}-v_{\textrm{L}}^{\textrm{b}},\quad i_{\textrm{A}}^{\textrm{ab}} = i_{\textrm{L}}^{\textrm{a}}-i_{\textrm{L}}^{\textrm{b}}.
\end{align*}
Here $\psi^{\textrm{ab}}=[1,-1,0]$. KVL gives
\begin{subequations}
\begin{align}
v_{\textrm{L}}^{\textrm{a}}-v_{\textrm{L}}^{\textrm{b}}&=z_{\textrm{A}}^{\textrm{ab}}\left(i_{\textrm{L}}^{\textrm{a}}-i_{\textrm{L}}^{\textrm{b}}\right) \\
&= \underset{=v_{\textrm{T}}^{\textrm{ab}}}{\underbrace{m_zz\left(i_{\textrm{L}}^{\textrm{a}}-i_{\textrm{L}}^{\textrm{b}}\right)}} + \underset{=v_{\textrm{F}}^{\textrm{ab}}}{\underbrace{\frac{m_rr_{\textrm{F}}}{2} \left(i_{\textrm{L}}^{\textrm{a}} -i_{\textrm{L}}^{\textrm{b}} + i_{\textrm{R}}^{\textrm{a}}- i_{\textrm{R}}^{\textrm{b}}\right)}}.\label{eq:llkvl}
\end{align}
Dividing through, we have
\label{eq:zab}
\begin{align}
z_{\textrm{A}}^{\textrm{ab}}&= \underset{=z_{\textrm{T}}}{\underbrace{m_zz}}+ \underset{=z_{\textrm{F}}^{\textrm{ab}}}{\underbrace{\frac{m_rr_{\textrm{F}}}{2} \left(1 + \frac{i_{\textrm{R}}^{\textrm{a}}- i_{\textrm{R}}^{\textrm{b}}}{i_{\textrm{L}}^{\textrm{a}}-i_{\textrm{L}}^{\textrm{b}}}\right)}}.\label{eq:zab1}
\end{align}
\end{subequations}

\subsubsection{LLG and symmetrical faults}\label{sec:llg}

LLL faults contain LL fault loops, and LLG and LLLG faults both contain LL and LG fault loops, and so can be modeled as in the previous two sections. LLG faults, for instance, can take several forms, depending on if there are resistances between the phases and to ground. The resistance between the phases is typically smaller than that to ground~\cite{ziegler2011numerical}; if we assume it is zero, the analysis is identical to that in Section~\ref{sec:ll} with $r_{\textrm{F}}=0$. If there is resistance between the phases, we can again model the LL fault loop, or the LG fault loop as in Section~\ref{sec:lg}. LLL and LLLG faults similarly admit the analyses of earlier sections. As such, we do not discuss them further here.

\subsubsection{Unfaulted phases}\label{sec:nofault}

We now discuss phases without faults. This applies both to scenario N in $\mathbb{F}$, in which there is no fault, and unfaulted phases during asymmetrical faults, e.g., phases b and c during an ag fault.

The apparent quantities are voltage and current vectors measured by the relay,
\begin{align*}
v_{\textrm{A}}^{\textrm{N}} &= v_{\textrm{L}},\quad i_{\textrm{A}}^{\textrm{N}} = i_{\textrm{L}}.
\end{align*}
Here, $\psi^{\textrm{N}}=\bm{I}$, the identity matrix. Ohm's law across, e.g., phase a while unfaulted gives
\begin{align*}
v_{\textrm{L}}^{a} &=  z i_{\textrm{L}}^{a}  + v_{\textrm{R}}^{a}.
\end{align*}

Note that traditionally, distance protection does not incorporate models of unfaulted phases because there is no need; other than close-in faults (which we can exclude by setting $\underbar{m}_z>0$), there is little risk of an unfaulted phase appearing faulted to the relay~\cite{horowitz2022power}. In other words, it is sufficient for the relay to check if its measurements match the model of a fault, and not to explicitly rule out the lack thereof. Here, we explicitly model normal operation because we later constrain auxiliary signals to make each fault appear different from normal operation.

Modeling unfaulted phases in an asymmetrical fault could help by reducing the size of the model's uncertainty set. On the other hand, modeling unfaulted phases leaves room for model error and adds complexity. For these reasons, we do not model unfaulted phases in this paper.

\subsubsection{Intermediate infeeds}\label{sec:model:inin}

Intermediate infeeds are junctions between the relay and potential faults~\cite{ziegler2011numerical}. In Figs.~\ref{fig:ag} and \ref{fig:ab}, this would appear as a third bus connecting to the line somewhere between buses L and R. We do not fully address intermediate infeeds in this paper, as doing so would add excessive notation and repetitive analyses. Instead, in this section we model intermediate infeeds, and in Section~\ref{sec:unc:inin} we describe how they would affect the uncertainty sets we later construct.

Let $m_{\textrm{I}}\in(0,1)$ and $i_{\textrm{I}}$ be the intermediate infeed's normalized location and current. The relay knows $m_{\textrm{I}}$, but does not observe $i_{\textrm{I}}$. The intermediate infeed adds a voltage drop to the fault loop, the form of which depends on $m_{\textrm{I}}$.
\begin{itemize}
\item If $m_{\textrm{I}}\geq m_z$, the fault occurs between the relay and the intermediate infeed. The only part of the fault loop $i_{\textrm{I}}$ flows through is the fault. The additional voltage drops for ag and ab faults are respectively
\[
m_r r_{\textrm{F}} i_{\textrm{I}}^{\textrm{a}}\quad\textrm{and}\quad \frac{m_r r_{\textrm{F}}}{2}\left( i_{\textrm{I}}^{\textrm{a}}-i_{\textrm{I}}^{\textrm{b}}\right).
\]
\item If $m_z>m_{\textrm{I}}$, the fault occurs after the intermediate infeed. Then $i_{\textrm{I}}$ flows through both the fault and the portion of the line in the fault loop. The additional voltage drops for ag and ab faults are respectively
\[
(m_z-m_{\textrm{I}})z\left(i_{\textrm{I}}^{\textrm{a}}+ki_{\textrm{I}}^0\right) + m_r r_{\textrm{F}} i_{\textrm{I}}^{\textrm{a}}
\]
and
\[
(m_z-m_{\textrm{I}})z\left(i_{\textrm{I}}^{\textrm{a}}-i_{\textrm{I}}^{\textrm{b}}\right) + \frac{m_r r_{\textrm{F}}}{2}\left( i_{\textrm{I}}^{\textrm{a}}-i_{\textrm{I}}^{\textrm{b}}\right).
\]
\end{itemize}
Because the relay cannot know if a fault occurs before or after the infeed, it should account for both cases. To avoid introducing a kink, e.g., with a maximum, they should be treated as two separate models. Hence, if an intermediate infeed is present, this means creating two models for each of the eleven faults in $\mathbb{F}\setminus\textrm{N}$.

\section{Model uncertainty}\label{sec:uncertainty}

While taking measurements, a relay has limited information about potential faults and the surrounding power network. Over the next few sections, we will integrate this uncertainty into the fault loop models of Section~\ref{sec:faultloops}. We distinguish between two timeframes.
\begin{itemize}
\item \textit{Pre-test.} Before measurement, the relay does not know if there is a fault and, if so, its type, $\eta\in\mathbb{F}\setminus\textrm{N}$, location, $m_z$, and resistance, $m_r$, or the voltages and currents at the terminals, $v_{\textrm{R}}$, $i_{\textrm{R}}$, $v_{\textrm{L}}$, and $i_{\textrm{L}}$. We face pre-test uncertainty when designing the auxiliary signal, $\Delta$, in Section~\ref{sec:protection}.
\item \textit{Post-test.} After measurement, the relay knows $v_{\textrm{L}}$ and $i_{\textrm{L}}$. The rest of the above quantities remain uncertain. The relay faces post-test uncertainty when testing for faults. We intend for the post-test uncertainty sets to be used as relay characteristics.
\end{itemize}
Under any uncertainty, there are sets of possible values that the right hand sides of (\ref{eq:KVL}) can take on. We will use the Minkowski sum to characterize these sets, both pre-test and post-test.

We use the following notation for uncertainty sets.
\begin{itemize}
\item Calligraphic letters denote the uncertainty sets of physical quantities; e.g., $\mathcal{V}_{\textrm{F}}^{\eta}$ is the set of all values that $v_{\textrm{F}}$ can take on in scenario $\eta\in\mathbb{F}$.
\item Given a set $\mathcal{X}$ and scalar $\alpha$,
\[
\alpha \mathcal{X} = \left\{
\alpha x \;\left | \; x \in \mathcal{X} \right.
\right\}.
\]
\item The Minkowski sum of two sets, $\mathcal{X}$ and $\mathcal{Y}$, is
\[
\mathcal{X}\oplus\mathcal{Y}=\left\{x+y\;|\; x\in\mathcal{X},\;y\in\mathcal{Y}\right\}.
\]
\end{itemize}

For each $\eta\in\mathbb{F}$, we will characterize the Minkowski sums
\begin{subequations}
\label{eq:KVLMS}
\begin{align}
\mathcal{V}_{\textrm{A}}^{\eta} &= \mathcal{V}_{\textrm{T}}^{\eta} \oplus \mathcal{V}_{\textrm{F}}^{\eta}\label{eq:KVLMSv}\\
\mathcal{Z}_{\textrm{A}}^{\eta} &= \mathcal{Z}_{\textrm{T}} \oplus \mathcal{Z}_{\textrm{F}}^{\eta}\label{eq:KVLMSz}
\end{align}
\end{subequations}
for pre-test and post-test uncertainty. In words, $\mathcal{V}_{\textrm{A}}^{\eta}$ and $\mathcal{Z}_{\textrm{A}}^{\eta}$ are all possible apparent voltages and impedances the relay can observe in scenario $\eta$. This is a generalization of (\ref{eq:KVL}) in which each quantity resides in an uncertainty set.\footnote{We could express the uncertainty in other ways, e.g., in the space of $(v_{\textrm{L}},i_{\textrm{L}})$. All forms lead to the same lists of constraints, which we write out in Appendix~\ref{app:sets}. We use this form because it is interpretable in terms of KVL.}

The Minkowski sum is in general intractable~\cite{khachiyan2008generating}. This is not an issue here, as we do not compute it explicitly for the pre-test uncertainty sets, and the structure of the post-test uncertainty sets renders it tractable.

Post-test, $\mathcal{V}_{\textrm{T}}$ and $\mathcal{V}_{\textrm{F}}$ do not depend on any common uncertainty and have exact, convex representations. Pre-test, however, they both depend on the local current, $i_{\textrm{L}}$, which is uncertain. This coupling introduces nonconvex bilinearities. We derive convex approximations of the pre-test uncertainty sets in Section~\ref{sec:convex}.

We now characterize the aspects of the uncertainty sets that are common to all faults. For all fault types, pre- and post-test,
\begin{align}
 \mathcal{Z}_{\textrm{T}} &= \{m_zz\;|\;m_z\in[\underbar{m}_z,1]\}.\label{eq:ZT}
\end{align}
Post-test, $\mathcal{V}_{\textrm{T}}^{\eta}$ is similarly simple because $v_{\textrm{T}}=m_zzi_{\textrm{A}}$, and $i_{\textrm{A}}$ is known. Hence, post-test, $\mathcal{Z}_{\textrm{T}}$ and $\mathcal{V}_{\textrm{T}}^{\eta}$ are both line segments. Pre-test, $\mathcal{V}_{\textrm{T}}^{\eta}$ is more complicated because $i_{\textrm{A}}$ is uncertain. $\mathcal{V}_{\textrm{F}}^{\eta}$ and $\mathcal{Z}_{\textrm{F}}^{\eta}$ are also more complicated because they depend on $m_z$, $m_r$, and the remote current, $i_{\textrm{R}}$, which are always uncertain.

\begin{approximation}\label{approx:SMIBR}
In order to characterize $i_{\textrm{R}}$, we make the following assumptions about the grid in the few cycles following a fault.
\begin{itemize}
\item SGs behave like balanced voltage sources. Let $v^{\circ}_{\mathcal{S}}$ be the vector of pre-fault voltages of the SGs.
\item IBRs behave like balanced current sources. Let $i^{\circ}_{\mathcal{C}}$ be the vector of pre-fault currents of the IBRs.
\item The voltage and current extraction at each load bus have an affine relationship. This could, e.g., represent an impedance, or the Th\'{e}venin or Norton equivalent of an RLC circuit.
\end{itemize}
\end{approximation}
Note that these assumptions could be different. For instance, the pre-fault voltage angle of a grid-forming IBR might persist for a few cycles after a fault. To further simplify the problem, we make the following approximation about the terminal voltages and currents.
\begin{approximation}\label{approx:hatm}
For each $\eta\in\mathbb{F}$, $i_{\textrm{L}}$, $i_{\textrm{R}}$, $v_{\textrm{L}}$, and $v_{\textrm{R}}$ depend on nominal values of $m_z$ and $m_r$, e.g., $\hat{m}_z=\hat{m}_r=1/2$.
\end{approximation}

Together, Approximations~\ref{approx:SMIBR} and~\ref{approx:hatm} enable us to express the terminal voltages and currents as linear function of the sources. For scenario $\eta\in\mathbb{F}$,
\begin{align}
\label{eq:vrir}
\begin{split}
i_{\textrm{L}} &= \Gamma^{\eta} \begin{bmatrix}
i^{\circ}_{\mathcal{C}}\\v^{\circ}_{\mathcal{S}}\end{bmatrix},\quad
i_{\textrm{R}} = \Theta^{\eta} \begin{bmatrix}
i^{\circ}_{\mathcal{C}}\\v^{\circ}_{\mathcal{S}}
\end{bmatrix},\\
v_{\textrm{L}} &= \Phi^{\eta} \begin{bmatrix}
i^{\circ}_{\mathcal{C}}\\v^{\circ}_{\mathcal{S}}\end{bmatrix},\quad
v_{\textrm{R}} = \Psi^{\eta} \begin{bmatrix}
i^{\circ}_{\mathcal{C}}\\v^{\circ}_{\mathcal{S}}
\end{bmatrix}.
\end{split}
\end{align}
We show how to compute these matrices for given $\hat{m}_z$ and $\hat{m}_r$ in Appendix~\ref{app:netmat}.

Approximation~\ref{approx:hatm} implies that the terminal currents, $i_{\textrm{L}}$ and $i_{\textrm{R}}$, do not depend on $m_z$ or $m_r$. As a result, the fault current, $i_{\textrm{F}}$, and its uncertainty set, $\mathcal{I}_{\textrm{F}}^{\eta}$, do not depend on $m_z$ or $m_r$ either. We may thus write
\[
\mathcal{V}_{\textrm{F}}^{\eta} = \left\{ m_rr_{\textrm{F}}\mathcal{I}_{\textrm{F}}^{\eta}  \;|\; 
m_r\in\left[0,1\right]
\right\},
\]
and note that this relation holds both pre- and post-test. The below lemma will be key to efficient computation later in Section~\ref{sec:projection}.
\begin{lemma}\label{lemma:VF0}
If $\mathcal{I}_{\textrm{F}}^{\eta}$ is convex, then $\mathcal{V}_{\textrm{F}}^{\eta}$ is the convex hull of the set $r_{\textrm{F}}\mathcal{I}_{\textrm{F}}^{\eta} \cup \{0\}$.
\end{lemma}
\begin{proof}
The lemma follows from the fact that every point in $\mathcal{V}_{\textrm{F}}^{\eta}$ is on a line segment between 0 and a point in $r_{\textrm{F}}\mathcal{I}_{\textrm{F}}^{\eta}$.
\end{proof}


To characterize each $\mathcal{V}_{\textrm{F}}^{\eta}$ and $\mathcal{I}_{\textrm{F}}^{\eta}$, we must first give uncertainty sets for the terminal currents. We do so by adding uncertainty to $v^{\circ}_{\mathcal{S}}$ and $i^{\circ}_{\mathcal{C}}$ and then relating them to $i_{\textrm{L}}$ and $i_{\textrm{R}}$ through (\ref{eq:vrir}).

Let $\lambda$ be a noise vector that resides in the set $\Lambda$. Let 
\[
\mathcal{U} = \left\{\left.
\begin{bmatrix}
i^{\circ}_{\mathcal{C}}\\v^{\circ}_{\mathcal{S}}
\end{bmatrix} + \Sigma\lambda\; \right| \; \lambda\in\Lambda
\right\},
\]
the uncertainty set of the SG voltage and IBR current sources, where the matrix $\Sigma$ scales $\lambda$ according to how uncertain the sources are. There are several forms the set $\Lambda$ can take. Here we consider the following two.
\begin{itemize}
\item Norm-bounded uncertainty:
\begin{align*}
\Lambda&=\left\{\lambda \;|\;  \left\|\lambda\right\|_2\leq 1\right\}.
\end{align*}
\item Polyhedral uncertainty, in which $\Lambda$ is a zonotope. Zonotopes are a special case of polytopes, and thus can be written in half-space form as
\begin{align*}
\Lambda&=\left\{\lambda \;|\;  P\lambda\leq \bm{1}\right\}.
\end{align*}
Let $P_k$ be row $k$ of $P$. Because zonotopes have $180^{\circ}$ rotational symmetry about the origin, $-P_k\lambda\leq 1$ is also in $\Lambda$, which implies $\lambda^{\top}P_k^{\top}P_k\lambda\leq 1$. Appendix~\ref{app:zono} provides further background on zonotopes.
\end{itemize}

The uncertainty sets for the terminal currents are
\begin{subequations}
\label{eq:vrirlambda}
\begin{align*}
\mathcal{I}_{\textrm{L}}^{\eta} &= \left\{\left.
  \Gamma^{\eta} u\; \right| \; u\in\mathcal{U}
\right\}\quad\textrm{and}\quad \mathcal{I}_{\textrm{R}}^{\eta} = \left\{\left.
  \Theta^{\eta} u\; \right| \; u\in\mathcal{U}
\right\}.
\end{align*}
\end{subequations}
If $\Lambda$ is a zonotope, then so are $\mathcal{U}$, $\mathcal{I}_{\textrm{L}}^{\eta}$, and $\mathcal{I}_{\textrm{R}}^{\eta}$ because affine transformations of zonotopes are zonotopes. If $\Lambda$ is norm-bounded, then these sets are ellipsoidal because affine transformations of ellipsoids are ellipsoids.

We will encounter products of $m_z$ and $m_r$ with the terminal currents, which result in products with $\lambda$. For a generic $m\in[\underbar{m},1]$, define the set
\begin{align*}
\begin{split}
\mathcal{U}_m &= \left\{
m \mathcal{U} \; \left| \; m\in[\underbar{m},1]
\right.\right\}\\
&= \left\{\left.
m\begin{bmatrix}
i^{\circ}_{\mathcal{C}}\\v^{\circ}_{\mathcal{S}}
\end{bmatrix} + m\Sigma\lambda\; \right| \; \lambda\in\Lambda,\;m\in[\underbar{m},1]\right\}.
\end{split}
\end{align*}
The term $m\Sigma\lambda$ is bilinear. Post-test in Section~\ref{sec:expost}, we encounter this bilinearity in the set $\mathcal{U}_{m_r}$ and are able to linearize it with a simple change of variables. Pre-test in Section~\ref{sec:exante}, we encounter bilinearities in $\mathcal{U}_{m_z}$ and $\mathcal{U}_{m_r}$. Because they share variables, we cannot linearize and instead derive convex relaxations in Section~\ref{sec:convex}.

We now make a few comments about how we have modeled the network and uncertainty. The logic in present-day distance relays is not based on explicit network modeling. Here, we must account for the network to capture how the auxiliary signal, which is added to the IBR currents, affects the relay's observations. We arrived at this model because it is based on first principles, accommodates standard assumptions about the behaviors of SGs and IBRs during faults, and is either linear or convex in the uncertain variables.

That said, we regard the relay's model of the network and sources as cruder than the relay's model of the line. This is because the relay knows precisely $z$ and the uncertainty sets of $m_z$ and $m_r$, whereas $i^{\circ}_{\mathcal{C}}$, $v^{\circ}_{\mathcal{S}}$, and $\Lambda$ are rough estimates of quantities elsewhere in the system. For this reason, we prioritize local over network-level uncertainty in our models. In particular, we could model the coupling between the voltage at the local terminal and at the fault in terms of the entire network, but prefer KVL on the fault loop; and we could include the constraint $v_{\textrm{L}}\in\mathcal{V}_{\textrm{L}}$, but do not because it is based on low-confidence, network-level modeling.

\section{Post-test uncertainty sets}\label{sec:expost}

We denote post-test uncertainty sets with an arrow, e.g., $\vec{\mathcal{V}}_{\textrm{L}}^{\eta}$. Today, a distance relay will conclude that there is a fault on its line if $z_{\textrm{A}}$ is in its characteristic, which is a bounded region of the complex plane. By construction, $\vec{\mathcal{Z}}_{\textrm{A}}^{\eta}$ is the smallest characteristic that contains every value of $z_{\textrm{A}}$ corresponding to scenario $\eta$. In this section, we characterize $\vec{\mathcal{Z}}_{\textrm{A}}^{\eta}$ for each $\eta\in\mathbb{F}\setminus\textrm{N}$ and derive analytical zonogon relaxations.

Post-test, the knowledge of $i_{\textrm{L}}$ simplifies the uncertainty sets. First, it implies that $\vec{\mathcal{V}}_{\textrm{A}}^{\eta}=i_{\textrm{A}}^{\eta}\vec{\mathcal{Z}}_{\textrm{A}}^{\eta}$, and hence that $z_{\textrm{A}}^{\eta}\in\vec{\mathcal{Z}}_{\textrm{A}}^{\eta}$ and $v_{\textrm{A}}^{\eta}\in\vec{\mathcal{V}}_{\textrm{A}}^{\eta}$ are equivalent criteria for each $\eta\in\mathbb{F}$. As such, we do not write out the corresponding voltage sets.

The only bilinearity we encounter post-test is between $m_r$ and $i_{\textrm{R}}$. We can write
\begin{align*}
\mathcal{U}_{m_r} &= \left\{\left.
m_r\begin{bmatrix}
i^{\circ}_{\mathcal{C}}\\v^{\circ}_{\mathcal{S}}
\end{bmatrix} + \Sigma\lambda_{m_r}\; \right| \; \lambda_{m_r}\in\Lambda_{m_r},\;{m_r}\in[0,1]\right\},
\end{align*}
where for norm-bounded uncertainty,
\begin{align*}
\Lambda_{m_r} &= \left\{\lambda_{m_r} \;|\;  \left\|\lambda_{m_r}\right\|_2\leq {m_r},\;{m_r}\in[0,1]\right\},
\end{align*}
and for zonotopic uncertainty,
\begin{align*}
\Lambda_{m_r} &= \left\{\lambda_{m_r} \;|\;  P\lambda_{m_r}\leq {m_r}\bm{1},\;{m_r}\in[0,1]\right\}.
\end{align*}
Because $\Lambda_{m_r}$ is convex in both cases, $\mathcal{U}_{m_r}$ and the uncertainty sets we derive below are also convex.

For any fault $\eta\in\mathbb{F}\setminus\textrm{N}$, we can write either (\ref{eq:lgkvlz}) or (\ref{eq:zab1}) in the form
\begin{subequations}
\begin{align}
z_{\textrm{A}}^{\eta} &= m_zz +  m_r\underset{=r_{\textrm{F}}i_{\textrm{F}}/i_{\textrm{A}}^{\eta}}{\underbrace{ \left(\xi^{\eta}  + \Xi^{\eta} i_{\textrm{R}} \right) }}\label{eq:zxi1}\\
&= m_zz + m_r \xi^{\eta}  + \Xi^{\eta} \Theta^{\eta}u_r,\label{eq:zxi2}
\end{align}
\end{subequations}
where
\begin{align*}
\xi^{\textrm{ag}} &= \frac{r_{\textrm{F}} i_{\textrm{L}}^{\textrm{a}}}{i_{\textrm{L}}^{\textrm{a}}+ki_{\textrm{L}}^0},\quad \Xi^{\textrm{ag}} = \frac{r_{\textrm{F}}}{i_{\textrm{L}}^{\textrm{a}}+ki_{\textrm{L}}^0}\psi^{\textrm{ag}},\\
\xi^{\textrm{ab}} &= \frac{r_{\textrm{F}} }{2},\quad\textrm{and}\quad\Xi^{\textrm{ab}} = \frac{r_{\textrm{F}}}{2\left(i_{\textrm{L}}^{a} -i_{\textrm{L}}^{b}\right)}\psi^{\textrm{ab}},
\end{align*}
and where $\psi^{\textrm{ag}}$ and $\psi^{\textrm{ab}}$ are as in Section~\ref{sec:faultloops}. The post-test uncertainty set for the fault impedance is 
\[
\vec{\mathcal{Z}}_{\textrm{F}}^{\eta}= \left\{\left. m_r\xi^{\eta}   + \Xi^{\eta}\Theta^{\eta}u_r \;\right|\; m_r\in[0,1],\;u_r\in \mathcal{U}_{m_r}  \right\}.
\]
This is convex because $\mathcal{U}_{m_r}$ is convex. Also recall the characterization of $\vec{\mathcal{Z}}_{\textrm{F}}^{\eta}$ in Lemma~\ref{lemma:VF0}, which we make use of for computation in the next section. The post-test uncertainty set for the apparent impedance seen by the relay is 
\[
\vec{\mathcal{Z}}_{\textrm{A}}^{\eta} = \mathcal{Z}_{\textrm{T}} \oplus \vec{\mathcal{Z}}_{\textrm{F}}^{\eta},
\]
which is convex because it is the Minkowski sum of two convex sets.

\subsection{Projection onto $\mathbb{C}$}\label{sec:projection}

A relay determines if there is a fault of type $\eta\in\mathbb{F}\setminus\textrm{N}$ by checking if $z_{\textrm{A}}^{\eta}\in\vec{\mathcal{Z}}_{\textrm{A}}^{\eta}$. As written above, this means searching for values of $m_r$, $u_r$, and $\lambda_{m_r}$ that are feasible with $z_{\textrm{A}}^{\eta}$. This can be done, e.g., by solving a linear or convex program with a zero objective.

While viable, this sort of computation is heavier than what modern distance relays solve online, i.e., checking if $z_{\textrm{A}}^{\eta}$ satisfies the handful of inequalities that make up its characteristic in $\mathbb{C}$. We therefore want to represent each $\vec{\mathcal{Z}}_{\textrm{A}}^{\eta}$ as a region of $\mathbb{C}$, with $z_{\textrm{A}}^{\eta}$ the only variable. This would enable a relay to evaluate $z_{\textrm{A}}^{\eta}\in\vec{\mathcal{Z}}_{\textrm{A}}^{\eta}$ in the same manner as is currently done.

To obtain such a representation, we must project $\vec{\mathcal{Z}}_{\textrm{A}}^{\eta}$ onto $\mathbb{C}$. A naive approach is to start with the above high-dimensional H-representation of $\vec{\mathcal{Z}}_{\textrm{A}}^{\eta}$, convert it to V-representation, project each vertex onto $\mathbb{C}$, and take the convex hull. This is inefficient because converting a polytope from H- to V-representation is NP-hard~\cite{khachiyan2008generating}.

We can efficiently project $\vec{\mathcal{Z}}_{\textrm{A}}^{\eta}$ onto $\mathbb{C}$ if $\Lambda$ is a zonotope. Denote the center and generators of $\Lambda$ by $c$ and $g_i$, $i=1,...,p$. The following steps produce the vertices of the projection of $\vec{\mathcal{Z}}_{\textrm{A}}^{\eta}$ onto $\mathbb{C}$.
\begin{enumerate}
\item The set $\frac{r_{\textrm{F}}}{i_{\textrm{A}}^{\eta}}\mathcal{I}^{\eta}_{\textrm{F}} = \xi^{\eta}  + \Xi^{\eta} \mathcal{I}^{\eta}_{\textrm{R}}$ is a zonogon with center
\begin{align}
c_{\textrm{F}}&=\xi^{\eta} + \Xi^{\eta}\Theta^{\eta}\left(\begin{bmatrix}
i^{\circ}_{\mathcal{C}}\\v^{\circ}_{\mathcal{S}}
\end{bmatrix} + \Sigma c\right)\label{ex:cF}
\end{align}
and generators $\Xi^{\eta}\Theta^{\eta}\Sigma g_i$, $i=1,...,p$. Convert this to V-representation using the algorithm in Appendix~\ref{app:zono}.
\item Following Lemma~\ref{lemma:VF0}, add the vertex $0$ to $r_{\textrm{F}}\mathcal{I}^{\eta}_{\textrm{F}}$ obtain the V-representation of $\vec{\mathcal{V}}_{\textrm{F}}^{\eta}$. We the have $\vec{\mathcal{Z}}_{\textrm{F}}^{\eta}=\frac{1}{i_{\textrm{A}}^{\eta}}\vec{\mathcal{V}}_{\textrm{F}}^{\eta}$.
\item Compute $\mathcal{Z}_{\textrm{T}} \oplus \vec{\mathcal{Z}}_{\textrm{F}}^{\eta}$. This amounts to adding the two vertices in $\mathcal{Z}_{\textrm{T}}$, $\left\{\underbar{m}_zz,z\right\}$, to each of the vertices in $\vec{\mathcal{Z}}_{\textrm{F}}^{\eta}$.
\item Take the convex hull, e.g., using the gift wrapping algorithm~\cite{jarvis1973identification}.
\end{enumerate}
The slowest part of this procedure is taking the convex hull; the gift wrapping algorithm has $\mathcal{O}\left(nh\right)$ complexity if $h$ of the $n$ vertices are on the hull.

\subsection{Zonogon relaxation}\label{sec:zonogon}

We now derive an analytical zonogon relaxation, which can be projected onto $\mathbb{C}$ even more efficiently than $\vec{\mathcal{Z}}_{\textrm{A}}^{\eta}$. Appendix~\ref{app:zono} provides background on zonogons, including their G-representation and conversion to V-representation.

Recall from Lemma~\ref{lemma:VF0} that $\vec{\mathcal{V}}_{\textrm{F}}^{\eta}$ is the convex hull of $0$ and $r_{\textrm{F}}\mathcal{I}^{\eta}_{\textrm{F}}$, and thus $\vec{\mathcal{Z}}_{\textrm{F}}^{\eta}$ of $0$ and $\frac{r_{\textrm{F}}}{i_{\textrm{A}}^{\eta}}\mathcal{I}^{\eta}_{\textrm{F}}$. Let $\hat{\mathcal{I}}^{\eta}_{\textrm{F}}$ be a shifted version of $\mathcal{I}^{\eta}_{\textrm{F}}$ with zero as its center, and let $\hat{\mathcal{Z}}_{\textrm{F}}^{\eta}$ be the convex hull of $\frac{r_{\textrm{F}}}{i_{\textrm{A}}^{\eta}}\left(\mathcal{I}^{\eta}_{\textrm{F}}\cup \hat{\mathcal{I}}^{\eta}_{\textrm{F}}\right)$. The zonogon relaxation is
\[
\hat{\mathcal{Z}}_{\textrm{A}}^{\eta} = \mathcal{Z}_{\textrm{T}} \oplus \hat{\mathcal{Z}}_{\textrm{F}}^{\eta}.
\]
\begin{lemma}\label{lemma:etazono1}
$\vec{\mathcal{Z}}_{\textrm{A}}^{\eta}\subseteq\hat{\mathcal{Z}}_{\textrm{A}}^{\eta}$.
\end{lemma}
\begin{proof}
$\vec{\mathcal{Z}}_{\textrm{F}}^{\eta}\subseteq\hat{\mathcal{Z}}_{\textrm{F}}^{\eta}$ because $0\in\hat{\mathcal{I}}_{\textrm{F}}^{\eta}$, which implies the lemma.
\end{proof}

The next lemma gives the analytical representation of $\hat{\mathcal{Z}}_{\textrm{A}}^{\eta}$. Recall $c_{\textrm{F}}$ from (\ref{ex:cF}).

\begin{lemma}\label{lemma:etazono2}
Suppose $\Lambda$ is a zonotope with center $c$ and generators $g_i$, $i=1,...,p$. Then $\hat{\mathcal{Z}}_{\textrm{A}}^{\eta}$ is a zonogon with center
\[
\frac{\left(1+\underbar{m}_z\right)z+c_{\textrm{F}}}{2}
\]
and generators 
\[
\frac{\left(1-\underbar{m}_z\right)z}{2},\;\frac{c_{\textrm{F}}}{2},\;\Xi^{\eta}\Theta^{\eta}\Sigma g_i,\; i=1,...,p.
\]
\end{lemma}
\begin{proof}
By inspection, $\mathcal{Z}_{\textrm{F}}^{\eta}$ has center $c_{\textrm{F}}/2$ and generators $c_{\textrm{F}}/2$ and $\Xi^{\eta}\Theta^{\eta}\Sigma g_i$, $i=1,...,p$. $\mathcal{Z}_{\textrm{T}}$ is a line segment with center $\left(1+\underbar{m}_z\right)z/2$ and generator $\left(1-\underbar{m}_z\right)z/2$. We obtain the lemma by calculating $\hat{\mathcal{Z}}_{\textrm{A}}^{\eta}=\mathcal{Z}_{\textrm{T}} \oplus \hat{\mathcal{Z}}_{\textrm{F}}^{\eta}$.
\end{proof}

In general, we want relaxations to be as tight as possible. We can tighten $\hat{\mathcal{Z}}_{\textrm{A}}^{\eta}$ by adding cuts in the complex plane. Intuitively, if $\Lambda$ is not too large, the relaxation allows for more uncertainty on the left side of $\hat{\mathcal{Z}}_{\textrm{A}}^{\eta}$. One way to reduce this uncertainty is by adding the cut $\textrm{Im}[z]\textrm{Re}[z_{\textrm{A}}^{\eta}]\geq \textrm{Re}[z]\textrm{Im}[z_{\textrm{A}}^{\eta}]$. This removes the portion of the zonogon to the left of the line in $\mathbb{C}$ defined by $z$, and is illustrated in Section~\ref{sec:example:postest}. Note, however, that this cut is not valid if $\Lambda$ is large. In this case, the uncertainty from the sources is larger than that due to the fault resistance, and the cut might pass though rather than alongside $\vec{\mathcal{Z}}_{\textrm{A}}^{\eta}$.

\subsection{Example}\label{sec:example:postest}
We plot $\vec{\mathcal{Z}}_{\textrm{A}}^{\eta}$ and $\hat{\mathcal{Z}}_{\textrm{A}}^{\eta}$ for faults $\eta\in\{\textrm{ag},\textrm{ab}\}$. These are representative of all faults because the LG (and LL) faults all have similar post-test uncertainty sets; and, as mentioned in Section~\ref{sec:faultloops}, the corresponding fault loops are present in all other types of faults. All quantities are per unit.

The test system is based on the IEEE 14-bus network, as given in Chapter 3 of \cite{baeckeland2022thesis}. The line admittances are from Table 3.1 therein. The buses with SGs are $\mathcal{S}=\{1, 3, 5, 14\}$, with IBRs $\mathcal{C}=\{2, 4, 7,8,12\}$, and with loads $\{6,10,11,13\}$. The SGs are represented as voltage sources with unit magnitude and angles evenly spaced in $[0,2\pi/10]$. The IBRs are current sources with unit magnitude and angles evenly spaced in $[0,2\pi/3]$, which reflects IBRs' greater degree of angle uncertainty during faults. All loads have admittance $0.1+j0.01$.

To each source $i_k$ or $v_k$, $k\in\mathcal{C}\cup\mathcal{S}$, we add balanced noise $\lambda_k\left[1,\alpha,\alpha^2\right]^{\top}$, $\lambda_k\in\Lambda_k\subset\mathbb{C}$. Each $\Lambda_k$ is a regular polygon in $\mathbb{C}$ with $n_{\varphi}\geq4$ sides. $n_{\varphi}$ must be even for $\Lambda_k$ to be a zonogon; $n_{\varphi}=4$ corresponds to a square, $n_{\varphi}=6$ a hexagon, and so on. Here, we used $n_{\varphi}=20$. This makes the post-test uncertainty sets appear more rounded than conventional relay characteristics.

The H-representation of each $\Lambda_k$ is
\[
\textrm{Re}\left[
e^{-j\varphi}\lambda_i
\right]\leq 1,\; \varphi=\frac{2\pi l}{n_{\varphi}},\; l=1,...,n_{\varphi}.
\]
This approaches the unit circle as $n_{\varphi}$ increases. The corresponding G-representation of $\Lambda_k$ has center 0 and generators
\[
g_l=\tan\left(\frac{\pi}{n_{\varphi}}\right)\begin{bmatrix}\sin(\varphi_l) \\ \cos(\varphi_l)\end{bmatrix},\; \varphi_l=\frac{2\pi l}{n_{\varphi}},\; l=1,...,\frac{n_{\varphi}}{2}.
\]
$\Lambda$ is the Cartesian product of $\Lambda_k$, $k\in\mathcal{C}\cup\mathcal{S}$. It is a zonotope with $\frac{n_{\varphi}}{2}(|\mathcal{S}|+|\mathcal{C}|)$ generators, each formed by appropriately zero-padding each $g_l$. We scale the noise vector by $\Sigma=0.1\bm{I}$ before adding it to the sources.

There is a relay at Bus 2 on the line from Bus~2 to Bus~3. The line's zero-sequence compensation factor is $k=2$, a typical value~\cite{horowitz2022power}. The maximum fault resistance is $r_{\textrm{F}}=1$, and the close-in fault threshold is $\underbar{m}_z=0.15$. In calculating the network matrices, we use the nominal fault location and resistance $\hat{m}_z=\hat{m}_r=1/2$.

The computations were carried out in Python, and the figures were created with Matplotlib~\cite{hunter2007matplotlib}. We computed each $\vec{\mathcal{Z}}_{\textrm{A}}^{\eta}$ via the procedure in Section~\ref{sec:projection}. Computing $\hat{\mathcal{Z}}_{\textrm{A}}^{\eta}$ consists of finding the G-representation in $\mathbb{C}$ using Lemma~\ref{lemma:etazono2} and then converting it to V-representation using the algorithm in Appendix~\ref{app:zono}. Computing $\vec{\mathcal{Z}}_{\textrm{A}}^{\eta}$ for all eleven faults took about 0.9 seconds, and largely depends on the number of sides in the uncertainty sets, $n_{\varphi}$. Computing the zonogon relaxation, $\hat{\mathcal{Z}}_{\textrm{A}}^{\eta}$, for all eleven faults took about 1 millisecond.

Figure~\ref{fig:ZMinkowskiSum} diagrams the Minkowski sum, $\vec{\mathcal{Z}}_{\textrm{A}}^{\textrm{ab}}=\mathcal{Z}_{\textrm{T}}\oplus\vec{\mathcal{Z}}_{\textrm{F}}^{\textrm{ab}}$. Here $\mathcal{Z}_{\textrm{T}}$ is the leftmost line segment; $\frac{r_{\textrm{F}}}{i_{\textrm{A}}^{\eta}} \mathcal{I}_{\textrm{F}}^{\textrm{ab}}$ is the ellipse in the bottom right; and $\mathcal{Z}_{\textrm{F}}^{\textrm{ab}}$ the convex hull of $\frac{r_{\textrm{F}}}{i_{\textrm{A}}^{\textrm{ab}}} \mathcal{I}_{\textrm{F}}^{\textrm{ab}}$ and zero, as in Lemma~\ref{lemma:VF0}.

\begin{figure}[h]
		\centering
\includegraphics[width=\columnwidth]{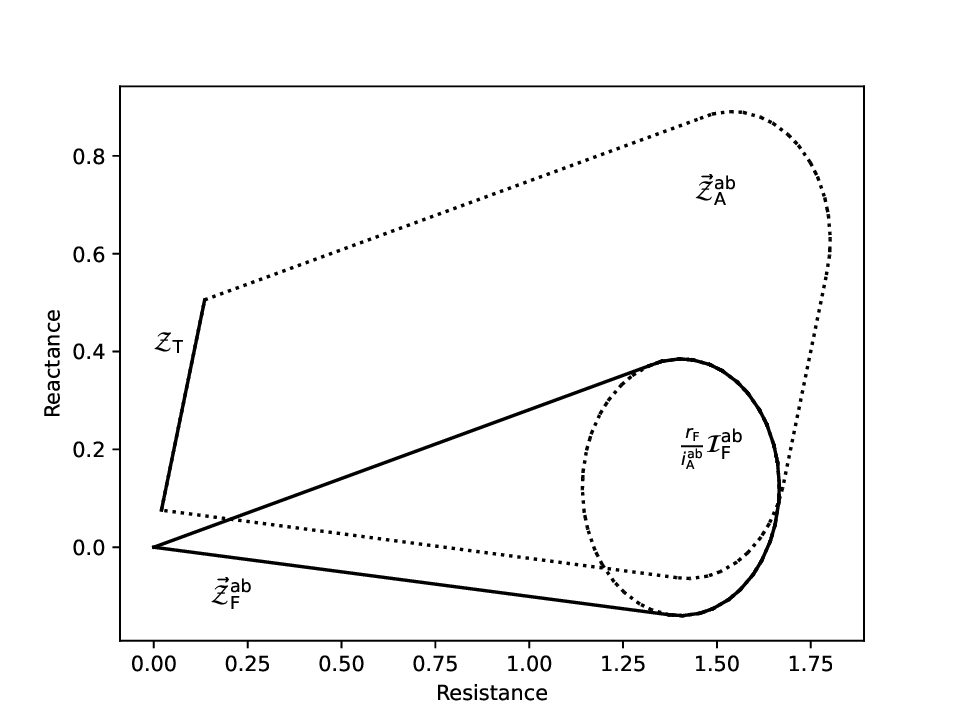}
	\caption{Diagram of the Minkowski sum, $\vec{\mathcal{Z}}_{\textrm{A}}^{\textrm{ab}}=\mathcal{Z}_{\textrm{T}}\oplus\vec{\mathcal{Z}}_{\textrm{F}}^{\textrm{ab}}$.}
	\label{fig:ZMinkowskiSum}
\end{figure}

Figure~\ref{fig:ZSets} shows $\vec{\mathcal{Z}}_{\textrm{A}}^{\eta}$, its zonogon relaxation, $\hat{\mathcal{Z}}_{\textrm{A}}^{\eta}$, and the cut $\textrm{Im}[z]\textrm{Re}[z_{\textrm{A}}^{\eta}]\geq \textrm{Re}[z]\textrm{Im}[z_{\textrm{A}}^{\eta}]$ for ag and ab faults.

\begin{figure}[h]
		\centering
\includegraphics[width=\columnwidth]{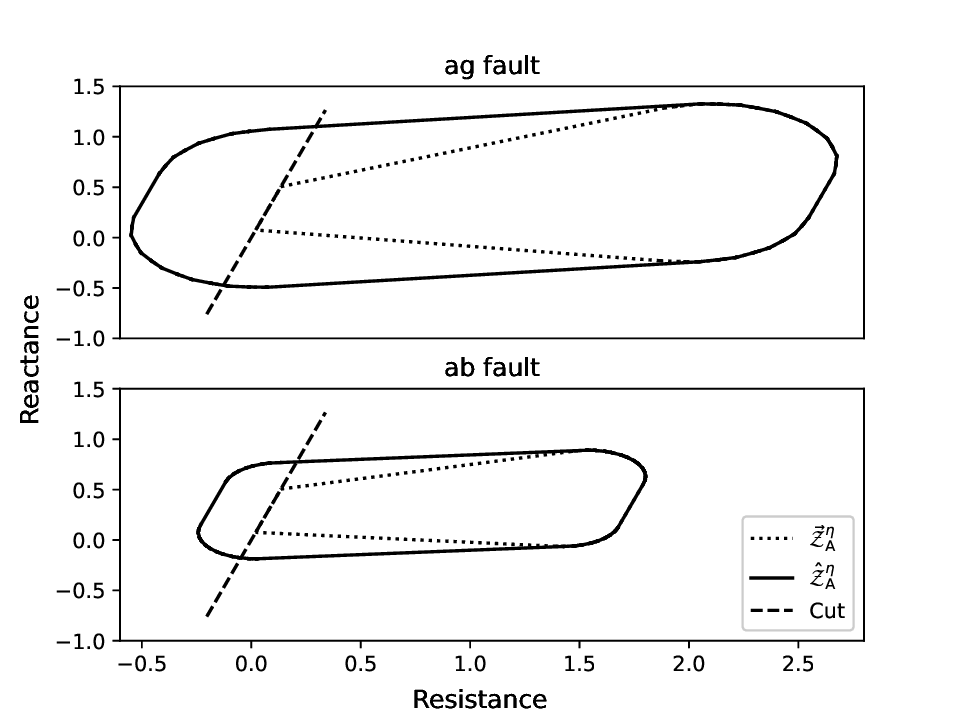}
	\caption{For ag and ab faults, the post-test uncertainty sets, their zonogon relaxations, and the cut $\textrm{Im}[z]\textrm{Re}[z_{\textrm{A}}^{\eta}]\geq \textrm{Re}[z]\textrm{Im}[z_{\textrm{A}}^{\eta}]$.}
	\label{fig:ZSets}
\end{figure}

\subsection{Relation to quadrilateral characteristics}\label{sec:quadrilateral}

We now discuss how our post-test uncertainty sets relate to the quadrilateral relay characteristic, which is considered most suitable for IBR-rich grids~\cite{kasztenny2022distance}. The quadrilateral characteristic is motivated by adding the uncertainty of each term in (\ref{eq:zxi2})~\cite{kasztenny2008fundamentals,ziegler2011numerical}. This is done in part by Minkowski sum, although, to our knowledge, it has never been referred to as such. The main idea is to add the uncertainty of the first two terms, $\mathcal{Z}_{\textrm{T}}\oplus \left\{\left.m_r\xi^{\eta} \right|m_r\in[0,1]\right\}$, which produces a parallelogram in $\mathbb{C}$. An ad hoc adjustment is then added for the third term in (\ref{eq:zxi2}). This adjustment is predominantly real because in SG-dominated grids, the voltages at the line terminals tend to have similar angles shortly after a fault. This assumption does not hold in IBR-dominated grids, potentially leading to misoperation~\cite{banaiemoqadam2019control,baeckeland2022distance}.

We may view the quadrilateral characteristic (and any other relay characteristic) as a type of post-test uncertainty set. $\vec{\mathcal{Z}}_{\textrm{A}}^{\eta}$ is the exact Minkowski sum of the post-test uncertainty, which we thus expect to be more accurate than a quadrilateral characteristic (if the network model is accurate). $\vec{\mathcal{Z}}_{\textrm{A}}^{\eta}$ is moreover an adaptive relay characteristics~\cite{horowitz2022power}, in that it depends on $i_{\textrm{L}}$ through $\xi^{\eta}$ and $\Xi^{\eta}$, and so must be recomputed online. As seen in Section~\ref{sec:example:postest}, this can be done efficiently, particularly for the zonogon relaxations, which can be computed in under a millisecond. If we do want to compute the characteristic offline, we could use an approximate, nominal value of $i_{\textrm{L}}$ in $\xi^{\eta}$ and $\Xi^{\eta}$.

A key difference is that in our uncertainty sets, we explicitly model how the terminal currents depend on the network and sources. This allows for any mix of SGs and IBRs, and enables us to systematically design auxiliary signals in Section~\ref{sec:auxiliary}. It is the formalization in terms of the Minkowski sum that enables us to integrate this additional level of detail.

\subsection{Underreach and overreach}

Underreach is a false negative error in which $m_z\in \left[\underbar{m}_z,1\right]$ and $z_{\textrm{A}}=m_zz+z_{\textrm{F}} \notin \vec{\mathcal{Z}}_{\textrm{A}}$, leading the relay to incorrectly conclude that the fault is beyond its line. Overreach is a false positive error in which $m_z > 1$ and $z_{\textrm{A}}=m_zz+z_{\textrm{F}} \in \vec{\mathcal{Z}}_{\textrm{A}}$, leading the relay to incorrectly conclude that there is a fault on its line. They are unavoidable as long as there is uncertainty in both fault location and the voltage drop across the fault.

To reduce error, a distance relay typically uses multiple characteristics, together known as a grading plan~\cite{ziegler2011numerical}. For instance, it might first check a smaller one, for which overreach is unlikely, and then a larger one, for which underreach is unlikely. In this way, it first checks for faults that are unambiguously on its line, allowing relays on other lines to do the same, and then checks for faults that could be on its own or an adjacent line.

The uncertainty sets we have constructed here allow for overreach, but not underreach. This is because we have assumed that $m_z\in\left[\underbar{m}_z,1\right]$---if the fault is anywhere on the relay's line (and not close-in), the impedance it measures must be in $\vec{\mathcal{Z}}_{\textrm{A}}$.

We can alternatively construct characteristics for which underreach can occur, but not overreach, by assuming that $m_z\in\left[\underbar{m}_z,\bar{m}_z\right]$, where $\bar{m}_z<1$. We then want $\bar{m}_z$ to be as large as possible so as to shrink the set of outcomes leading to underreach. This value is given by minimizing $\bar{m}_z$ subject to $z\in\vec{\mathcal{Z}}_{\textrm{A}}$, i.e., that the line impedance is in the characteristic---if $\bar{m}_z$ is any larger, we allow overreach, and if it is any smaller, we ignore realizations of $z_{\textrm{A}}$ that unambiguously correspond to $m_z\in\left[\underbar{m}_z,1\right]$.

\section{Pre-test uncertainty sets}\label{sec:exante}

Pre-test, $i_{\textrm{L}}$ or $v_{\textrm{L}}$ are variables because the relay has not yet observed them. We will use pre-test uncertainty sets to optimize auxiliary signals in Section~\ref{sec:auxiliary}. We hereon indicate pre-test uncertainty sets with a breve. We characterize the voltage uncertainty set, $\breve{\mathcal{V}}_{\textrm{A}}^{\eta}$ in (\ref{eq:KVLMSv}) for each $\eta\in\mathbb{F}$.

We do not characterize the pre-test impedance uncertainty set, $\breve{\mathcal{Z}}_{\textrm{A}}^{\eta}$, for two reasons. Because the local currents, $i_{\textrm{L}}$, are now uncertain, $z_{\textrm{F}}$ contains a nonlinear, rational term, which is difficult to analyze; and, impedances are not basic variables. As a result, e.g., $z_{\textrm{A}}^{\textrm{ag}}$ is not explicitly coupled to $z_{\textrm{A}}^{\textrm{ab}}$, despite both depending on $v^{\textrm{a}}_{\textrm{L}}$. These characteristics make impedance uncertainty sets poorly suited for optimizing auxiliary signals over multiple fault models.

For any fault $\eta\in\mathbb{F}\setminus \textrm{N}$, we can write (\ref{eq:KVLv}) in the form
\begin{align}
\psi^{\eta}v_{\textrm{L}} &= \Omega_z^{\eta}u_z + \Omega_r^{\eta}u_r,\label{eq:vlomega}
\end{align}
where $\Omega_z^{\eta}$ and $\Omega_r^{\eta}$ are matrices that we will compute later, and $u_z$ and $u_r$ represent the products $m_zu$ and $m_ru$, respectively. When there is no fault, we have $v_{\textrm{L}} = \Omega^{\textrm{N}}u$, and will similarly compute $\Omega^{\textrm{N}}$ later. For $\eta\in\mathbb{F}\setminus \textrm{N}$, the pre-test uncertainty sets for the line and fault voltage drops in (\ref{eq:KVLMSv}) are
\begin{align*}
\breve{\mathcal{V}}_{\textrm{T}}^{\eta} &= \{\Omega_z^{\eta}u_z\;\left|\; u_z\in\mathcal{U}_{m_z}\right.\}\\
\breve{\mathcal{V}}_{\textrm{F}}^{\eta}
&= \{\Omega_r^{\eta}u_r \;\left|\;  u_r\in\mathcal{U}_{m_r} \right.\}.
\end{align*}

$\mathcal{U}_{m_z}$ and $\mathcal{U}_{m_r}$ both contain bilinear terms. To contain them in one place, define
\begin{align*}
\mathcal{W}&=\left\{ (u_z,u_r)\;|\;u_z\in\mathcal{U}_{m_z},\;u_r\in\mathcal{U}_{m_r} \right\}\\
&= \left\{
(u_z,u_r)\;|\;u_z=m_z\begin{bmatrix}
i^{\circ}_{\mathcal{C}}\\v^{\circ}_{\mathcal{S}}
\end{bmatrix} + m_z\Sigma\lambda,\right.\\
&\hspace{22mm} u_r=m_r\begin{bmatrix}
i^{\circ}_{\mathcal{C}}\\v^{\circ}_{\mathcal{S}}
\end{bmatrix} + m_r\Sigma\lambda,\\
&\hspace{22mm}\left.\lambda\in\Lambda,m_z\in\left[\underbar{m}_z,1\right],m_r\in[0,1]\right\}.
\end{align*}
Unlike the post-test uncertainty in Section~\ref{sec:expost}, we cannot eliminate the bilinearities in $\mathcal{W}$ with a change of variables. We will construct convex approximations of $\mathcal{W}$ in Section~\ref{sec:convex}.

The pre-test uncertainty set for the apparent voltage in scenario $\eta\in\mathbb{F}\setminus\textrm{N}$ is
\begin{align*}
\breve{\mathcal{V}}_{\textrm{A}}^{\eta} &= \breve{\mathcal{V}}_{\textrm{T}}^{\eta} \oplus \breve{\mathcal{V}}_{\textrm{F}}^{\eta}\\
&= \left\{\left.\Omega_z^{\eta}u_z + \Omega_r^{\eta}u_r \;\right|\; (u_z,u_r)\in\mathcal{W} \right\}.
\end{align*}
In Appendix~\ref{app:sets}, we write the pre-test uncertainty sets, their intersections, and their duals as explicit lists of constraints.

\subsection{LG faults}\label{sec:pre:LG}

In (\ref{eq:lgkvl}), we wrote KVL for the phase a-to-ground fault loop, which consisted of the voltage drops across the line and across the fault. The pre-test uncertainty set for the line voltage drop is
\begin{align*}
\breve{\mathcal{V}}_{\textrm{T}}^{\textrm{ag}} &=  \{m_zz\left(i_{\textrm{L}}^{\textrm{a}}+ki_{\textrm{L}}^0\right)\;\left|\;m_z\in[\underbar{m}_z,1],\; i_{\textrm{L}}\in \mathcal{I}^{\textrm{ag}}_{\textrm{L}}\right.\}\\
&= \{\underset{=\Omega_z^{\textrm{ag}}}{\underbrace{z\left(e^{\textrm{a}}+ke^0\right)\Gamma^{\textrm{ag}}}}u_z\;\left|\; u_z\in\mathcal{U}_{m_z}\right.\}.
\end{align*}
The uncertainty set of the fault voltage drop is
\begin{align*}
\breve{\mathcal{V}}_{\textrm{F}}^{\textrm{ag}}
&= \left\{\left.m_rr_{\textrm{F}}\left(i_{\textrm{L}}^{\textrm{a}}+ i_{\textrm{R}}^{\textrm{a}}\right) \;\right|\;  m_r\in\left[0,1\right] ,\; i_{\textrm{L}}\in \mathcal{I}^{\textrm{ag}}_{\textrm{L}}  ,\; i_{\textrm{R}}\in \mathcal{I}^{\textrm{ag}}_{\textrm{R}}  \right\}\\
&= \{\underset{=\Omega_r^{\textrm{ag}}}{\underbrace{r_{\textrm{F}}e^{\textrm{a}}\left(\Gamma^{\textrm{ag}}+ \Theta^{\textrm{ag}}\right)}}u_r \;\left|\;  u_r\in\mathcal{U}_{m_r} \right.\}.
\end{align*}

\subsection{LL faults}
In (\ref{eq:llkvl}), we wrote KVL for the phase a-to-phase b fault loop. The pre-test uncertainty sets for the line and fault voltage drops are
\begin{align*}
\breve{\mathcal{V}}_{\textrm{T}}^{\textrm{ab}} &=  \{\left.m_zz\left(i_{\textrm{L}}^{\textrm{a}}-i_{\textrm{L}}^{\textrm{b}}\right)\;\right|\;m_z\in[\underbar{m}_z,1],\; i_{\textrm{L}}\in \mathcal{I}^{\textrm{ab}}_{\textrm{L}}\}\\
&= \{\underset{=\Omega_z^{\textrm{ab}}}{\underbrace{z\left(e^{\textrm{a}}-e^{\textrm{b}}\right)\Gamma^{\textrm{ab}}}}u_z\;\left|\; u_z\in\mathcal{U}_{m_z}\right.\}\\
\breve{\mathcal{V}}_{\textrm{F}}^{\textrm{ab}} &= \left\{\left.(m_rr_{\textrm{F}}/2)\left(i_{\textrm{L}}^{\textrm{a}}-i_{\textrm{L}}^{\textrm{b}} + i_{\textrm{R}}^{\textrm{a}} -  i_{\textrm{R}}^{\textrm{b}}\right)  \;\right|\; \right.\\
&\quad\quad\left.m_r\in[0,1],\;i_{\textrm{L}}\in \mathcal{I}^{\textrm{ab}}_{\textrm{L}}  ,\; i_{\textrm{R}}\in \mathcal{I}^{\textrm{ab}}_{\textrm{R}}  \right\}\\
&= \{\underset{=\Omega_r^{\textrm{ab}}}{\underbrace{(r_{\textrm{F}}/2)\left(\left(e^{\textrm{a}}-e^{\textrm{b}}\right)\Gamma^{\textrm{ab}} + \left(e^{\textrm{a}} -  e^{\textrm{b}}\right)\Theta^{\textrm{ab}}\right)}}u_r  \;\left|\; u\in\mathcal{U}_{m_r}  \right.\}.
\end{align*}

\subsection{Normal operation}\label{sec:exante:unfaulted}

The pre-test uncertainty set for normal operation, which we modeled in Section~\ref{sec:nofault}, is
\begin{align*}
\breve{\mathcal{V}}_{\textrm{A}}^{\textrm{N}}  &= z\mathcal{I}^{\textrm{N}}_{\textrm{L}} + \mathcal{V}^{\textrm{N}}_{\textrm{R}} \\
&= \{\underset{=\Omega^{\textrm{N}}}{\underbrace{\left(z\Gamma^{\textrm{N}} +  \Psi^\textrm{N}\right)}}u  \;\left|\; u\in\mathcal{U}  \right.\}.
\end{align*}

\subsection{Intersections}\label{sec:intersection}

When we design auxiliary signals in Section~\ref{sec:protection}, we will make use of the intersections of pre-test uncertainty sets. We take these intersections in the space of the variables $v_{\textrm{L}}$ and $i_{\textrm{L}}$ because they are observed post-test (as opposed to, $m_z$, $m_r$, and $i_{\textrm{R}}$, which remain uncertain).

Consider the pre-test uncertainty sets $\breve{\mathcal{V}}_{\textrm{A}}^{\eta_1}$ and $\breve{\mathcal{V}}_{\textrm{A}}^{\eta_2}$, $(\eta_1,\eta_2)\in\mathbb{F}_2$. With a slight abuse of notation, we use superscripts to indicate to which set a variable belongs, e.g., $m_z^{\eta_1}$ and $m_z^{\eta_2}$ are variables for $\breve{\mathcal{V}}_{\textrm{A}}^{\eta_1}$ and $\breve{\mathcal{V}}_{\textrm{A}}^{\eta_2}$, respectively.

We take the intersection by setting $v_{\textrm{L}}^{\eta_1}=v_{\textrm{L}}^{\eta_2}$ and $i_{\textrm{L}}^{\eta_1}=i_{\textrm{L}}^{\eta_2}$. We apply the former using the relation $v_{\textrm{A}}^{\eta}=\psi^{\eta}v_{\textrm{L}}$, $\eta\in\{\eta_1,\eta_2\}$. Substituting, we rewrite the latter as
\begin{align}
\Gamma^{\eta_1}\left(\begin{bmatrix}
i^{\circ}_{\mathcal{C}}\\v^{\circ}_{\mathcal{S}}
\end{bmatrix} + \Sigma\lambda^{\eta_1}\right)&=\Gamma^{\eta_2}\left(\begin{bmatrix}
i^{\circ}_{\mathcal{C}}\\v^{\circ}_{\mathcal{S}}
\end{bmatrix} + \Sigma\lambda^{\eta_2}\right)\label{eq:iLiL},
\end{align}
which is linear. We denote intersection of $\breve{\mathcal{V}}_{\textrm{A}}^{\eta_1}$ and $\breve{\mathcal{V}}_{\textrm{A}}^{\eta_2}$ by
\begin{align*}
\breve{\mathcal{P}}^{\eta_1,\eta_2} &= \left\{v_{\textrm{L}} \;\left|\;
\psi^{\eta_1}v_{\textrm{L}} \in \breve{\mathcal{V}}_{\textrm{A}}^{\eta_1},\;
\psi^{\eta_2}v_{\textrm{L}} \in \breve{\mathcal{V}}_{\textrm{A}}^{\eta_2},(\ref{eq:iLiL})
\right.\right\}.\label{eq:Pinter}
\end{align*}

\subsection{Convex approximations}\label{sec:convex}

The pre-test uncertainty sets and their intersections are nonconvex because they depend on $\mathcal{W}$, which has nonconvex bilinearities. We give relaxations of $\mathcal{W}$ in Section~\ref{sec:convex:relax} and a restriction and approximation in Section~\ref{sec:restapp}. In general, if we replace $\mathcal{W}$ with a relaxation (restriction), we obtain relaxations (restrictions) of $\breve{\mathcal{V}}_{\textrm{A}}^{\eta}$, $\eta\in\mathbb{F}$. We comment that relaxations may be most appropriate in this context, as they correspond to larger, and hence more conservative uncertainty sets. We discuss intersections in Section~\ref{sec:interconvex}.

We use the following relaxations because they are simple, and note there are other choices, e.g., McCormick relaxations~\cite{mccormick1976computability} and the semidefinite relaxation of our prior work~\cite{taylor2024fault}.

\subsubsection{Three relaxations}\label{sec:convex:relax}
Observe that for a generic $m\in[\underbar{m},1]$, we can write 
\begin{align*}
\mathcal{U}_m &= \left\{\left.
m\begin{bmatrix}
i^{\circ}_{\mathcal{C}}\\v^{\circ}_{\mathcal{S}}
\end{bmatrix} + \Sigma\lambda_m\; \right| \; \lambda_m\in\Lambda_m,\;m\in[\underbar{m},1]\right\}\\
\Lambda_m &= \left\{m\Lambda \;|\;m\in[\underbar{m},1]\right\}.
\end{align*}

We obtain a simple relaxation by defining
\begin{align*}
\mathcal{U}_m^{\textrm{Rel}} &= \left\{\left.
m\begin{bmatrix}
i^{\circ}_{\mathcal{C}}\\v^{\circ}_{\mathcal{S}}
\end{bmatrix} + \Sigma\lambda\; \right| \; \lambda\in\Lambda,\;m\in[\underbar{m},1]\right\},
\end{align*}
which is convex. Here we have replaced the term $\Lambda_m$ in $\mathcal{U}_m$ with $\Lambda$. Let
\[
\mathcal{W}^{\textrm{Rel},1} = \left\{ (u_z,u_r) \;|\; u_z\in\mathcal{U}_{m_z}^{\textrm{Rel}},\;u_r\in\mathcal{U}_{m_r}^{\textrm{Rel}}\right\}.
\]
Observe that $\Lambda_m\subseteq\Lambda$ for any $m\in[\underbar{m},1]$. This implies that $\mathcal{U}_m\subseteq\mathcal{U}_m^{\textrm{Rel}}$, which implies the below lemma.
\begin{lemma}
$\mathcal{W}\subseteq\mathcal{W}^{\textrm{Rel},1}$.
\end{lemma}

We now derive a second convex relaxation. Under ellipsoidal uncertainty, $\Lambda_m$ has the form
\begin{align*}
\Lambda_m &= \left\{m\lambda \;|\;  \left\|\lambda\right\|_2\leq 1,\;m\in[\underbar{m},1]\right\}\\
&= \left\{\lambda_m \;|\;  \left\|\lambda_m\right\|_2\leq m,\;m\in[\underbar{m},1]\right\},
\end{align*}
and under zonotopic uncertainty,
\begin{align*}
\Lambda_m &= \left\{m\lambda \;|\;  P\lambda\leq \bm{1},\;m\in[\underbar{m},1]\right\}\\
&= \left\{\lambda_m \;|\;  P\lambda_m\leq m\bm{1},\;m\in[\underbar{m},1]\right\}.
\end{align*}
For ellipsoidal or zonotopic uncertainty, $\mathcal{U}_m$ and $\Lambda_m$ are convex when written this way. However, to use this representation for $\Lambda_{m_z}$ and $\Lambda_{m_r}$ simultaneously, we must add a nonconvex constraint to ensure that both sets depend on the same underlying uncertainty, $\lambda\in\Lambda$.

To obtain a convex relaxation, we express this representation of $\mathcal{W}$ (with norm-bounded uncertainty) as the below list of constraints.
\begin{subequations}
\begin{align}
&m_z\in[\underbar{m}_z,1],m_r\in[0,1]\label{eq:mzmr}\\
&u_z = m_z\begin{bmatrix}
i^{\circ}_{\mathcal{C}}\\v^{\circ}_{\mathcal{S}}
\end{bmatrix} + \Sigma\lambda_{m_z}\\
&u_r = m_r\begin{bmatrix}
i^{\circ}_{\mathcal{C}}\\v^{\circ}_{\mathcal{S}}
\end{bmatrix} + \Sigma\lambda_{m_r}\\
&\left\|\lambda_{m_z}\right\|_2 \leq m_z\label{eq:uz}\\
&\left\|\lambda_{m_r}\right\|_2 \leq m_r\label{eq:ur}\\
&\frac{1}{m_z}\lambda_{m_z}=\frac{1}{m_r}\lambda_{m_r}.\label{eq:bllambda}
\end{align}
\end{subequations}
Note that with polyhedral instead of ellipsoidal uncertainty, we replace (\ref{eq:uz}) and (\ref{eq:ur}) with the linear constraints $P\lambda_{m_z} \leq m_z\bm{1}$ and $P\lambda_{m_r} \leq m_r\bm{1}$. In either case, the only nonconvexity is (\ref{eq:bllambda}), which serves the role of ensuring that $\lambda$ is the same in both $\Lambda_{m_z}$ and $\Lambda_{m_r}$. Define
\begin{align*}
\mathcal{W}^{\textrm{Rel},2} = \left\{ (u_z,u_r) \;|\; (u_z,u_r) \textrm{ are feasible for (\ref{eq:mzmr})-(\ref{eq:ur})}\right\}.
\end{align*}
$\mathcal{W}^{\textrm{Rel},2}$ is $\mathcal{W}$ without the nonconvex constraint, (\ref{eq:bllambda}), which implies the following result.
\begin{lemma}
$\mathcal{W}\subseteq \mathcal{W}^{\textrm{Rel},2}$.
\end{lemma}

We obtain a tighter convex relaxation by taking the intersection of $\mathcal{W}^{\textrm{Rel},1}$ and $\mathcal{W}^{\textrm{Rel},2}$:
\[
\mathcal{W}^{\textrm{Rel},3} = \mathcal{W}^{\textrm{Rel},1}\cap\mathcal{W}^{\textrm{Rel},2}.
\]

\subsubsection{Restriction and approximation}\label{sec:restapp}
We can write $\mathcal{W}$ (with norm-bounded uncertainty) yet another (equivalent) way as
\begin{subequations}
\begin{align}
&m_{z,1},m_{z,2}\in[\underbar{m}_z,1],\\
&m_{r,1},m_{r,2}\in[0,1]\\
&m_{z,1}=m_{z,2},\;m_{r,1}=m_{r,2}\label{eq:m1m2}\\
&u_z = m_{z,1}\begin{bmatrix}
i^{\circ}_{\mathcal{C}}\\v^{\circ}_{\mathcal{S}}
\end{bmatrix} + \sqrt{\frac{m_{z,2}}{m_{r,2}}}\Sigma\lambda_m\label{eq:mzr1}\\
&u_r = m_{r,1}\begin{bmatrix}
i^{\circ}_{\mathcal{C}}\\v^{\circ}_{\mathcal{S}}
\end{bmatrix} + \sqrt{\frac{m_{r,2}}{m_{z,2}}}\Sigma\lambda_m\label{eq:mzr2}\\
&\left\|\lambda_m\right\|_2 \leq \sqrt{m_{z,1}m_{r,1}}.\label{eq:soc}
\end{align}
For zonotopic uncertainty, we replace (\ref{eq:soc}) with
\begin{align}
P\lambda_m\leq \sqrt{m_{z,1}m_{r,1}}\bm{1}.\label{eq:soc1}
\end{align}
\end{subequations}
The repeated variables in this formulation allow us to construct the following two approximations.
\begin{itemize}
\item We obtain a convex restriction by adding the constraint $m_{r,2}=m_{z,2}$, so that (\ref{eq:mzr1}) and (\ref{eq:mzr2}) become linear. We denote this $\mathcal{W}^{\textrm{Res}}$. Clearly, $\mathcal{W}^{\textrm{Res}}\subseteq\mathcal{W}$.
\item We obtain a convex approximation by removing (\ref{eq:m1m2}) from $\mathcal{W}^{\textrm{Res}}$. In this case, (\ref{eq:soc}) is a hyperbolic constraint, which can be written as a second-order cone constraint~\cite{Boyd1998SOCP}. Constraint (\ref{eq:soc1}) can also be written as collection of hyperbolic constraints by squaring each row. This approximation, which we denote $\mathcal{W}^{\textrm{SOC}}$, is neither a relaxation or a restriction because it involves adding and removing constraints.
\end{itemize}

\subsubsection{Summary of inclusions}
The relationships between the convex approximations in this section are summarized the inclusions
\[
\mathcal{W}^{\textrm{Res}} \subseteq \mathcal{W}   \subseteq \mathcal{W}^{\textrm{Rel},3}  \subseteq\left\{\mathcal{W}^{\textrm{Rel},1},\mathcal{W}^{\textrm{Rel},2}\right\},
\]
and the following lemma.
\begin{lemma}\label{lemma:WSOC}
$\mathcal{W}^{\textrm{Res}}\subseteq\mathcal{W}^{\textrm{SOC}}\subseteq\mathcal{W}^{\textrm{Rel},1}$.
\end{lemma}
\begin{proof}
The first inclusion follow from the fact that we obtain $\mathcal{W}^{\textrm{SOC}}$ by removing constraints from $\mathcal{W}^{\textrm{Res}}$. The second inclusion follows from the fact that we obtain $\mathcal{W}^{\textrm{Rel},1}$ by replacing (\ref{eq:soc}) (or  (\ref{eq:soc1})) in $\mathcal{W}^{\textrm{SOC}}$ with $\lambda\in\Lambda$, a looser constraint.
\end{proof}

\subsubsection{Intersections}\label{sec:interconvex}

We now formulate convex approximations of intersections of pre-test uncertainty sets, $\breve{\mathcal{P}}^{\eta_1,\eta_2}$, $(\eta_1,\eta_2)\in\mathbb{F}_2$, as described in Section~\ref{sec:intersection}. This is akin to joining the constraints of the two sets and adding the constraint (\ref{eq:iLiL}). Difficulty arises because the variables $\lambda^{\eta_1}$ and $\lambda^{\eta_2}$ in (\ref{eq:iLiL}) are not present in some of the approximations. We describe how each convex approximation affects intersections below. We remark that this section is not meant to be comprehensive, as there are a number of ways to mix and match convex approximations.

$\lambda^{\eta_1}$ and $\lambda^{\eta_2}$ are present in $\mathcal{W}^{\textrm{Rel},1}$, but not $\mathcal{W}^{\textrm{Rel},2}$. As a result, we can straightforwardly take intersections of relaxations based on $\mathcal{W}^{\textrm{Rel},1}$, and we can use $\mathcal{W}^{\textrm{Rel},2}$ to strengthen these relaxations. We show this explicitly for an example in Appendix~\ref{app:intersections}.

We obtain a convex restriction of $\breve{\mathcal{P}}^{\eta_1,\eta_2}$ by similarly using $\mathcal{W}^{\textrm{Res}}$ and, instead of (\ref{eq:iLiL}), the restriction
\begin{align*}
\Gamma^{\eta_1}u_z^{\eta_1} &=\Gamma^{\eta_2}u_z^{\eta_2}\\
\Gamma^{\eta_1}u_m^{\eta_1} &=\Gamma^{\eta_2}u_m^{\eta_2}.
\end{align*}

Finally, we obtain a convex approximation of $\breve{\mathcal{P}}^{\eta_1,\eta_2}$ using $\mathcal{W}^{\textrm{SOC}}$ and, instead of (\ref{eq:iLiL}), the approximation
\begin{align}
\Gamma^{\eta_1}\left(\begin{bmatrix}
i^{\circ}_{\mathcal{C}}\\v^{\circ}_{\mathcal{S}}
\end{bmatrix} + \Sigma\lambda^{\eta_1}_m\right)&=\Gamma^{\eta_2}\left(\begin{bmatrix}
i^{\circ}_{\mathcal{C}}\\v^{\circ}_{\mathcal{S}}
\end{bmatrix} + \Sigma\lambda^{\eta_2}_m\right).
\end{align}

\subsection{Intermediate infeeds}\label{sec:unc:inin}

We now describe how an intermediate infeed, as described in Section~\ref{sec:model:inin}, would affect the pre- and post-test uncertainty sets. Later in Footnote~\ref{fn:inin}, we describe how to account for intermediate infeeds in auxiliary signal design. Recall that to account for an intermediate infeed, we must create two models for each fault in $\mathbb{F}\setminus\textrm{N}$: one with $m_{\textrm{I}}\geq m_z$ (fault between relay and infeed), and one with $m_{\textrm{I}}< m_z$ (infeed between relay and fault).

For the eleven fault models with $m_{\textrm{I}}\geq m_z$, the intermediate infeed results in a bilinearity of the form $m_r\lambda$. Because this bilinearity is already present pre- and post-test, it does not significantly change the structure of the uncertainty sets---the convexification strategies from Section~\ref{sec:convex} apply as before to the pre-test uncertainty sets in Section~\ref{sec:exante}, and the post-test uncertainty sets in Section~\ref{sec:expost} remain linear or ellipsoidal (depending on $\Lambda$). This is because, in other words, we can simply add the current, $i_{\textrm{I}}$, to the remote infeed current, $i_{\textrm{R}}$, and proceed as before.

For the eleven fault models with $m_{\textrm{I}}< m_z$, the intermediate infeed results in bilinearities of the form  $m_z\lambda$ and  $m_r\lambda$. These are both already present in the pre-test uncertainty sets, and so the convexification strategies from Section~\ref{sec:convex} apply as before. However, instead of one bilinearity, the post-test uncertainty sets now have the same two bilinearities as the pre-test uncertainty sets. As a result, they no longer have linear or ellipsoidal exact reformulations. Instead, we could use the techniques from Section~\ref{sec:convex} to obtain convex approximations of the post-test uncertainty sets.

\section{Distance protection}\label{sec:protection}

The relay's task is to determine which element of $\mathbb{F}$ is true. Today, this is done by checking if $z_{\textrm{A}}^{\eta}$ is in the corresponding characteristic for each $\eta\in\mathbb{F}\setminus\textrm{N}$. The uncertainty sets in Section~\ref{sec:expost} serve as relay characteristics in this paper. For a given model $\eta\in\mathbb{F}\setminus\textrm{N}$, the pre-test uncertainty set is the union of all post-test uncertainty sets that can be realized upon observation of $i_{\textrm{L}}$. The purpose of an auxiliary signal is to ensure that only the correct element of $\mathbb{F}$ appears true to the relay. Here, we formulate fault detection in terms of post-test uncertainty sets and auxiliary signal design in terms of the pre-test uncertainty sets of Section~\ref{sec:exante}. We proceed by adapting our prior work in~\cite{taylor2024fault}. As we are now designing auxiliary signals, we hereon explicitly indicate dependence on $\Delta$.





\subsection{Fault detection}\label{sec:consep}

We assume that one of the models in $\mathbb{F}$ is the correct one. Post-test, once the relay has measured $i_{\textrm{L}}$ and $v_{\textrm{L}}$, we want to know which one it is.

\begin{definition}
$i_{\textrm{L}}$ and $v_{\textrm{L}}$ are consistent with model $\eta\in\mathbb{F}$ if $z_{\textrm{A}}^{\eta}\in\vec{\mathcal{Z}}_{\textrm{A}}^{\eta}(\Delta)$.
\end{definition}
If a model is consistent with the observations, then the relay believes that the corresponding scenario could be true. Due to uncertainty, $i_{\textrm{L}}$ and $v_{\textrm{L}}$ could be consistent with multiple models.

We use the auxiliary signal, $\Delta$, to ensure that they only consistent with the correct model. We will do so by requiring that no pair of pre-test uncertainty sets is simultaneously feasible. Recall the from Section~\ref{sec:intersection} that we denote the intersection of $\breve{\mathcal{V}}_{\textrm{A}}^{\eta_1}(\Delta)$ and $\breve{\mathcal{V}}_{\textrm{A}}^{\eta_2}(\Delta)$, $(\eta_1,\eta_2)\in\mathbb{F}_2$, by $\breve{\mathcal{P}}^{\eta_1,\eta_2}(\Delta)$.

\begin{definition}
The auxiliary signal, $\Delta$, separates the pair $(\eta_1,\eta_2)\in\mathbb{F}_2$ if $\breve{\mathcal{P}}^{\eta_1,\eta_2}(\Delta)=\emptyset$.
\end{definition}

If two models are separated, then no observation, $i_{\textrm{L}}$ and $v_{\textrm{L}}$, can be consistent with both of them. In other words, if $\eta_1$ and $\eta_2$ are separated, then for any possible $i_{\textrm{L}}$ and $v_{\textrm{L}}$, we cannot have $z_{\textrm{A}}^{\eta_1}\in\vec{\mathcal{Z}}_{\textrm{A}}^{\eta_1}(\Delta)$ and $z_{\textrm{A}}^{\eta_2}\in\vec{\mathcal{Z}}_{\textrm{A}}^{\eta_2}(\Delta)$.

Recall that in Section~\ref{sec:interconvex}, we discussed convex relaxations and restrictions of intersections.
\begin{lemma}\label{lemma:relaxedseparation}
Given the pair $(\eta_1,\eta_2)\in\mathbb{F}_2$,
\begin{itemize}
\item if $\mathring{\mathcal{P}}^{\eta_1,\eta_2}(\Delta)\subseteq\breve{\mathcal{P}}^{\eta_1,\eta_2}(\Delta)$ and $\mathring{\mathcal{P}}^{\eta_1,\eta_2}(\Delta)\neq\emptyset$, then $\Delta$ does not separate $(\eta_1,\eta_2)$; and
\item if $\breve{\mathcal{P}}^{\eta_1,\eta_2}(\Delta)\subseteq\tilde{\mathcal{P}}^{\eta_1,\eta_2}(\Delta)$ and $\tilde{\mathcal{P}}^{\eta_1,\eta_2}(\Delta)=\emptyset$, then $\Delta$ separates $(\eta_1,\eta_2)$.
\end{itemize}
\end{lemma}

Lemma~\ref{lemma:relaxedseparation} implies that (i) if a restriction of $\breve{\mathcal{P}}^{\eta_1,\eta_2}(\Delta)$ is feasible, then the auxiliary signal does not achieve separation, and (ii) if a relaxation of $\breve{\mathcal{P}}^{\eta_1,\eta_2}(\Delta)$ is infeasible, then the auxiliary signal does achieve separation. We can directly test both conditions for a given auxiliary signal. In the next section, we use the second condition to optimize over separating auxiliary signals.

\subsection{Auxiliary signal design}\label{sec:auxiliary}

Let $Q$ be a positive definite matrix, and consider the below optimization.
\begin{subequations}
\label{opt:AS1}
\begin{align}
\min_{\Delta} \quad&\Delta^{\dag}Q\Delta\\
\textrm{such that}\quad & \breve{\mathcal{P}}^{\eta_1,\eta_2}(\Delta)=\emptyset,\quad (\eta_1,\eta_2)\in\mathbb{F}_2.\label{opt:intercon}
\end{align}
\end{subequations}
The objective is a proxy for the disruption caused by the auxiliary signal; if $Q=\bm{I}$, it is $\left\|\Delta\right\|_2^2$. The optimal solution is the minimally disruptive auxiliary signal that separates all models, so that the relay can identify which element of $\mathbb{F}$ is true.\footnote{If there were an intermediate infeed, as described in Sections~\ref{sec:model:inin} and~\ref{sec:unc:inin}, we would not enforce separation between two models corresponding to the same type of fault. This is because both models correspond to the same conclusion for the relay, and hence need not be separated.\label{fn:inin}}
The following is adaptated from Lemma 1 in~\cite{taylor2024fault}.
\begin{lemma}\label{lem:Farkas}
There exists a convex set, $\mathcal{S}^{\eta_1,\eta_2}(\Delta)$, the non-emptiness of which implies that $\Delta$ separates the pair $(\eta_1,\eta_2)\in\mathbb{F}_2$.
\end{lemma}
\begin{proof}
Let $\tilde{\mathcal{P}}^{\eta_1,\eta_2}(\Delta)$ be a convex relaxation of $\breve{\mathcal{P}}^{\eta_1,\eta_2}(\Delta)$. Farkas' lemma~\cite{Schrijver1998LPIP} states that there is another convex conic set, $\mathcal{S}^{\eta_1,\eta_2}(\Delta)$, which is nonempty if and only if $\tilde{\mathcal{P}}^{\eta_1,\eta_2}(\Delta)=\emptyset$. Therefore, by Lemma~\ref{lemma:relaxedseparation}, if $\mathcal{S}^{\eta_1,\eta_2}(\Delta)\neq\emptyset$, then $\Delta$ separates models $\eta_1$ and $\eta_2$.\footnote{Farkas' lemma technically only applies when the pre-test uncertainty sets are polyhedral. There are, however, extensions to convex and semidefinite systems~\cite{ben1969linear,klep2013exact}, which allow us to proceed in the same fashion. In this case, we are in effect assuming that some constraint qualification holds.}
\end{proof}

We refer to each $\mathcal{S}^{\eta_1,\eta_2}(\Delta)$ as a dual system. In Appendix~\ref{app:dualsystems}, we explain how to construct the dual systems and list the corresponding constraints for a particular polyhedral relaxation of $\breve{\mathcal{P}}^{\eta_1,\eta_2}(\Delta)$.

Constraint (\ref{opt:intercon}) is not well-suited for optimization. Lemma~\ref{lem:Farkas} provides a conventional constraint that is a sufficient condition for (\ref{opt:intercon}). Replacing (\ref{opt:intercon}), we obtain the following optimization.
\begin{subequations}
\label{opt:AS2}
\begin{align}
\min_{\Delta} \quad&\Delta^{\dag}Q\Delta\\
\textrm{such that}\quad & \mathcal{S}^{\eta_1,\eta_2}(\Delta)\neq\emptyset,\quad (\eta_1,\eta_2)\in\mathbb{F}_2.\label{opt:Farkascon}
\end{align}
\end{subequations}
Constraint (\ref{opt:Farkascon}) is a convex restriction of (\ref{opt:intercon}). As a result, the optimal objective of (\ref{opt:AS2}) will be greater than or equal to that of (\ref{opt:AS1}), and the resulting auxiliary signal will separate all $(\eta_1,\eta_2)\in\mathbb{F}_2$. 

\subsection{Analysis and computation}

Lemma~\ref{lemma:relaxedseparation} provides sufficient conditions for determining if an auxiliary signal does or does not separate a pair of models. These conditions amount testing the feasibility of linear or convex sets. Lemma~\ref{lem:Farkas} provides another sufficient condition for separation, which can be tested similarly.

The optimization (\ref{opt:AS2}) is biconvex. However, it is an offline computation and not necessarily large, e.g., if the network is small or has been reduced beforehand. We could thus solve (\ref{opt:AS2}) with a number of different nonlinear programming algorithms, e.g., the convex-concave procedure~\cite{taylor2024fault}. Here we use the routine in Appendix~\ref{app:admm}, which is an instance of the Alternating Direction Method of Multipliers~\cite{boyd2011distributed}.

We gain some qualitative insight by examining the dual systems in Appendix~\ref{app:dualsystems}. The following all shrink the size of each fault's uncertainty set:
\begin{itemize}
\item increasing $\underbar{m}_z$, the threshold for close-in faults;
\item decreasing the source uncertainty (by scaling $\Sigma$ down or $P$ up); and
\item setting $r_{\textrm{F}}=0$, i.e., only considering bolted faults. 
\end{itemize}
The first two bring the dual systems closer to feasibility, as can be seen by examining the terms $\underbar{m}_z\mu_{\textrm{zl}}^{\eta}$ and $\bm{1}^{\top}\phi^{\eta}$ in the leading inequality constraints. The third allows us to drop the dual variable $\mu_{\textrm{ru}}$ and several equality constraints, similarly bringing the dual systems closer to feasibility.

Now consider the leading inequality constraint in each dual system. The only terms that can be positive are the first, which captures the network, and $\underbar{m}_z\mu_{\textrm{zl}}^{\eta}$, the close-in term; without at least one of these, the dual systems would always be infeasible and separation impossible. By making these network terms positive, the auxiliary signal, $\Delta$, brings the dual systems to feasibility, achieving separation.

\subsection{Example}\label{sec:example:AS}

We first plot auxiliary signals in the complex plane in Section~\ref{sec:example:ASGrid}, and then optimize auxiliary signals in Section~\ref{sec:example:ASOpt}. We use the 14-bus system described in Section~\ref{sec:example:postest}. To make the problem more difficult, thus making auxiliary signals necessary, we increase the noise magnitude to $\Sigma = \bm{I}$ and make all sources IBRs, so that $\mathcal{C}=\{1,2,3,4,5,8,14\}$. In all cases, the auxiliary signals are negative-sequence current injections.

We use the relaxation $\mathcal{W}^{\textrm{Rel},3}$ from Section~\ref{sec:convex:relax} for the pre-test uncertainty sets. The constraints that make up the corresponding dual systems are listed in Section~\ref{app:dualsystems}.

All optimizations were carried out in Python using CVXPy~\cite{diamond2016cvxpy} and the solver Clarabel~\cite{goulart2024clarabel}. CVXPy was particularly suitable because it supports complex numbers.

\subsubsection{Sets of separating auxiliary signals}\label{sec:example:ASGrid}

We first have all IBRs inject the same negative sequence phasor, $\delta\in\mathbb{C}$, so that $\Delta=\delta\bm{1}\in\mathbb{C}^{|\mathcal{C}|}$. This enables us to plot sets of separating auxiliary signals. We do so by discretizing the complex plane and solving a feasibility problem at each gridpoint.

The left panel of Fig.~\ref{fig:ASGrid} shows the separating auxiliary signals for N and ag and ab faults. The smallest auxiliary signal (on the grid) that separate all pairs---$(\textrm{N},\textrm{ag})$, $(\textrm{N},\textrm{ab})$, and $(\textrm{ag},\textrm{ab})$---is $\delta= j 0.2$, for which $\left\|\Delta\right\|_2\approx0.53$. Within the range of the figure, all auxiliary signals that separate $(\textrm{N},\textrm{ab})$ also separate $(\textrm{N},\textrm{ag})$ and $(\textrm{ag},\textrm{ab})$.

The right panel of Fig.~\ref{fig:ASGrid} shows the separating auxiliary signals for N and all LG and LL faults. The smallest auxiliary signal on the grid that separates N and all LG faults has magnitude $\left\|\Delta\right\|_2\approx3.02$, and similarly for N and all LL faults. The smallest auxiliary signal on the grid that separates N and all LG and LL faults is $\delta=1.6 + j0.6$, for which $\left\|\Delta\right\|_2\approx4.52$.

\begin{figure}[h]
		\centering
\includegraphics[width=\columnwidth]{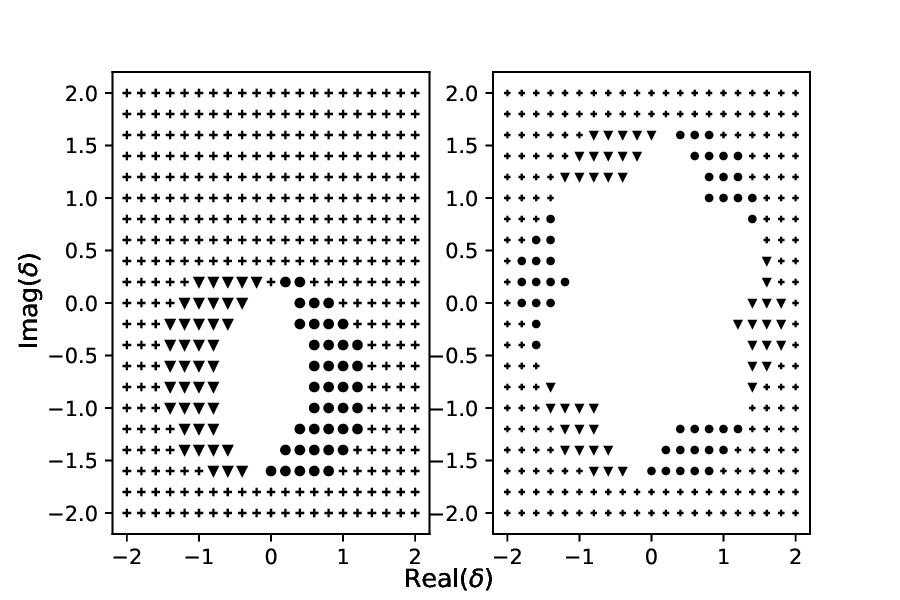}
	\caption{Left: separating auxiliary signals for the pair $(\textrm{N},\textrm{ag})$ ($\bullet $), the pair $(\textrm{N},\textrm{ab})$ ($\blacktriangledown $), and pairs $(\textrm{N},\textrm{ag})$, $(\textrm{N},\textrm{ab})$, and $(\textrm{ag},\textrm{ab})$ (+). Right: Separating auxiliary signals for all LG faults ($\bullet $), all LL faults ($\blacktriangledown$), and all LG and LL faults (+).}
	\label{fig:ASGrid}
\end{figure}


\subsubsection{Optimizing auxiliary signals}\label{sec:example:ASOpt}

We now use the ADMM routine in Appendix~\ref{app:admm} to optimize an auxiliary signal that separates N and all LL and LG faults. We allow each IBR to inject a different negative-sequence current, in hopes that the additional degrees of freedom will allow for a smaller auxiliary signal. We set $Q=\bm{I}$ and $\rho=1$. We tried two starting points: $\Delta_0=\bm{0}$ and $\Delta_0=(1.6 + j0.6)\bm{1}$, the latter of which corresponds to the best grid point in the right panel of Fig.~\ref{fig:ASGrid}. In both cases, the routine found the separating auxiliary signal
\begin{align*}
\Delta\approx\;&[1.55+j0.57 \quad 1.56+j0.56 \quad 1.59+j0.50  \\
& 1.50+j0.64 \quad 1.50+j0.63 \quad 1.52+j0.60\\
& 1.48+j0.66]^{\dag}.
\end{align*}
The magnitude is $\left\|\Delta\right\|_2\approx 4.35$, a 4\% reduction from the best grid point in the right panel of Fig.~\ref{fig:ASGrid}. This improvement could be due to allowing each IBR to set a different negative sequence phasor, or not restricting the auxiliary signal to a gridpoint. In either case, this indicates that additional degrees of freedom might not be helpful, as a simpler auxiliary signal achieves similar simulated performance.

\section{Conclusions and future work}

We have used the Minkowski sum to systematically aggregate the uncertainty faced by distance relays. This allows us to construct characteristics that explicitly account for whatever SG and IBR sources are in the surrounding power grid. IBRs can inject auxiliary signals to make a relay's detection problem unambiguous. We have used Farkas' lemma to construct an optimization for IBR auxiliary signals.

A clear next step in this research is evaluation of the new relay characteristics and optimized auxiliary signals in electromagnetic transient simulation. There are several potential methodological improvements. The network matrices depend on nominal values of the fault's location and resistance. It may be possible to improve on this approximation, e.g., by combining characteristics based on multiple nominal values. The auxiliary signal optimization can be straightforwardly extended to multiple relays on different lines by adding dual system constraints for each relay. It may be of interest to distribute this problem, e.g., if the IBRs each set their own auxiliary signals. 

We expect that our core ideas---aggregating uncertainty with the Minkowski sum and using Farkas' lemma to optimize parameters pre-test---are applicable to other schemes like overcurrent and time-domain protection.

\bibliographystyle{IEEEtran}
\bibliography{MainBib,JATBib}

\appendices
\section{Network matrices}\label{app:netmat}

We compute the matrices $\Gamma^{\eta}$, $\Theta^{\eta}$, $\Phi^{\eta}$, and $\Psi^{\eta}$ in (\ref{eq:vrir}) for each $\eta\in\mathbb{F}$. We divide the buses into IBRs, $\mathcal{C}$, SGs, $\mathcal{S}$, junctions, $\mathcal{J}$, and the virtual bus where the fault occurs, $\textrm{F}$. Junctions include buses with zero current injection and with impedance loads connected to them. All satisfy $i_{\mathcal{J}}=-Y_{\mathcal{J}}v_{\mathcal{J}}$, where the admittance is zero for physical junctions. Recall from Approximation~\ref{approx:hatm} that $\hat{m}_z$ and $\hat{m}_r$ are the fault's nominal normalized location and resistance.

The relation between $v_{\textrm{F}}$ and $i_{\textrm{F}}$ depends on the fault.
\begin{itemize}
\item A phase a to ground fault has $v_{\textrm{F}}^{\textrm{a}}=\hat{m}_rr_{\textrm{F}}i_{\textrm{F}}^{\textrm{a}}$, and $i_{\textrm{F}}^{\textrm{b}}=i_{\textrm{F}}^{\textrm{c}}=0$.
\item An ab fault has $v_{\textrm{F}}^{\textrm{a}}-v_{\textrm{F}}^{\textrm{b}}=\hat{m}_rr_{\textrm{F}}\left(i_{\textrm{F}}^{\textrm{a}}-i_{\textrm{F}}^{\textrm{b}}\right)$ and $i_{\textrm{F}}^{\textrm{c}}=0$.
\item When there is no fault, $Y_{\textrm{F}}^{\textrm{N}}=\bm{0}$.
\end{itemize}
In general, we can write $i_{\textrm{F}}=-Y_{\textrm{F}}^{\eta}v_{\textrm{F}}$.

Let $Y\in\mathbb{R}^{3(n+1)\times 3(n+1)}$ be the bus admittance matrix, which accounts for the virtual fault bus. Substituting for $i_{\textrm{F}}$ and $i_{\mathcal{J}}$, we have
\[
Y \begin{bmatrix}
v_{\textrm{F}}\\
v_{\mathcal{J}}\\
v_{\mathcal{C}}\\
v^{\circ}_{\mathcal{S}}
\end{bmatrix}=
\begin{bmatrix}
-Y_{\textrm{F}}^{\eta}v_{\textrm{F}}\\
-Y_{\mathcal{J}}v_{\mathcal{J}}\\
i^{\circ}_{\mathcal{C}}\\
i_{\mathcal{S}}
\end{bmatrix}.
\]
Let $Y_{\mathcal{S}}$ be the columns of $Y$ that multiply $v^{\circ}_{\mathcal{S}}$. Let $I_{\mathcal{S}}$ and $I_{\mathcal{C}}$ be thin matrices with an identity in the entries corresponding to the buses in $\mathcal{S}$ and $\mathcal{C}$, respectively, and zeros elsewhere. Then we have
\begin{align*}
&\underset{=Y_{\textrm{LHS}}^{\eta}}{\underbrace{\left(Y
+\begin{bmatrix}
Y_{\textrm{F}}^{\eta} &\bm{0} &\bm{0} & \multirow{4}{*}{$-I_{\mathcal{S}}-Y_{\mathcal{S}}$} \\
\bm{0} & Y_{\mathcal{J}} & \bm{0} & \\
\bm{0} & \bm{0}& \bm{0} & \\
\bm{0} & \bm{0} & \bm{0} & 
\end{bmatrix}\right)}}
\begin{bmatrix}
v_{\textrm{F}}\\
v_{\mathcal{J}}\\
v_{\mathcal{C}}\\
i_{\mathcal{S}}
\end{bmatrix}=\\
&\quad
\underset{=Y_{\textrm{RHS}}}{\underbrace{\begin{bmatrix}
I_{\mathcal{C}} & -Y_{\mathcal{S}}
\end{bmatrix}}}
\begin{bmatrix}
i^{\circ}_{\mathcal{C}}\\
v^{\circ}_{\mathcal{S}}
\end{bmatrix}
\end{align*}
Let $\tilde{Y}^{\eta}=\left(Y_{\textrm{LHS}}^{\eta}\right)^{-1}Y_{\textrm{RHS}}$. We obtain the voltage at bus $k\notin\mathcal{S}$ by premultiplying $\tilde{Y}^{\eta}\begin{bmatrix}
i^{\circ}_{\mathcal{C}}\\
v^{\circ}_{\mathcal{S}}
\end{bmatrix}$ as follows. Let $D_k$ and $E_k$ be such that
\[
v_k=D_k\begin{bmatrix}
v_{\textrm{F}}\\
v_{\mathcal{J}}\\
v_{\mathcal{C}}\\
i_{\mathcal{S}}
\end{bmatrix}\quad\textrm{and}\quad v_k=E_k\begin{bmatrix}
i^{\circ}_{\mathcal{C}}\\
v^{\circ}_{\mathcal{S}}
\end{bmatrix}
\]
for $k\notin\mathcal{S}$ and $k\in\mathcal{S}$, respectively. Let
\[
\aleph_{k}^{\eta} = \left\{\begin{array}{ll}
D_k\tilde{Y}^{\eta} & \textrm{if }k\notin\mathcal{S}\\
E_k & \textrm{if }k \in\mathcal{S}
\end{array}
\right..
\]
Then $v_k = \aleph_k^{\eta}\begin{bmatrix}
i^{\circ}_{\mathcal{C}}\\
v^{\circ}_{\mathcal{S}}
\end{bmatrix}$, and $\Phi^{\eta}=\aleph^{\eta}_{\textrm{L}}$ and $\Psi^{\eta}=\aleph^{\eta}_{\textrm{R}}$. Let
\begin{align*}
\Gamma^{\eta}&=\frac{1}{\hat{m}_zz}\left(\Phi-\aleph_{\textrm{F}}^{\eta} \right)\\
\Theta^{\eta}&=\frac{1}{\left(1-\hat{m}_z\right)z}\left(\Psi-\aleph_{\textrm{F}}^{\eta} \right).
\end{align*}
The terminal currents in scenario $\eta$ are then
\begin{align*}
i_{\textrm{L}} &= \frac{1}{\hat{m}_zz}\left(v_{\textrm{L}}-v_{\textrm{F}} \right)\\
&=\Gamma^{\eta}\begin{bmatrix}
i^{\circ}_{\mathcal{C}}\\
v^{\circ}_{\mathcal{S}}
\end{bmatrix}\\
i_{\textrm{R}} &= \frac{1}{(1-\hat{m}_z)z}\left(v_{\textrm{R}}-v_{\textrm{F}} \right)\\
&=\Theta^{\eta}\begin{bmatrix}
i^{\circ}_{\mathcal{C}}\\
v^{\circ}_{\mathcal{S}}
\end{bmatrix}.
\end{align*}

\section{Uncertainty sets}\label{app:sets}

\subsection{Pre-test}

The set $\mathcal{W}$ for fault $\eta\in\mathbb{F}\setminus\textrm{N}$ is given by
\begin{align*}
&m_z^{\eta}\in[0,1],\;m_r^{\eta}\in[0,1]\\
&u_z^{\eta} = m_z^{\eta}\begin{bmatrix}
i^{\circ}_{\mathcal{C}}\\v^{\circ}_{\mathcal{S}}
\end{bmatrix} + m_z^{\eta}\Sigma\lambda^{\eta}\\
&u_r^{\eta} = m_r^{\eta}\begin{bmatrix}
i^{\circ}_{\mathcal{C}}\\v^{\circ}_{\mathcal{S}}
\end{bmatrix} + m_r^{\eta}\Sigma\lambda^{\eta}\\
&\left\|\lambda^{\eta}\right\|_2^2\leq1 \quad\textrm{or}\quad P\lambda^{\eta}\leq\bm{1},
\end{align*}
depending on if the uncertainty is ellipsoidal or polyhedral.

For fault $\eta\in\mathbb{F}\setminus\textrm{N}$, the set $\breve{\mathcal{V}}_{\textrm{A}}^{\eta}$ consists of $\mathcal{W}$ and the constraint
\[
\psi^{\eta}v_{\textrm{L}} = \Omega_z^{\eta}u_z^{\eta} + \Omega_r^{\eta}u_r^{\eta},
\]
which we derived earlier (as (\ref{eq:vlomega}) in Section~\ref{sec:exante}), along with the coefficients $\Omega_z^{\eta}$, $\Omega_r^{\eta}$, and $\Omega^{\textrm{N}}$. $\breve{\mathcal{V}}_{\textrm{A}}^{\textrm{N}}$ consists of the constraints
\begin{align}
\breve{\mathcal{V}}_{\textrm{A}}^{N}:\quad&
\left\{\begin{array}{l}
v_{\textrm{L}} = \Omega^{\textrm{N}}u \\
u  = \begin{bmatrix}
i^{\circ}_{\mathcal{C}}\\v^{\circ}_{\mathcal{S}}
\end{bmatrix} + \Sigma\lambda^{\textrm{N}} \\
\left\|\lambda^{\textrm{N}} \right\|_2^2\leq1 \quad\textrm{or}\quad P\lambda^{\textrm{N}} \leq\bm{1}.
\end{array}\right.
\end{align}

\subsubsection{Intersections}\label{app:intersections}

We list the constraints that make up a convex relaxation of $\breve{\mathcal{P}}^{\eta_1,\eta_2}$, the intersection of $\breve{\mathcal{V}}_{\textrm{A}}^{\eta_1}$ and $\breve{\mathcal{V}}_{\textrm{A}}^{\eta_2}$, for $(\eta_1,\eta_2)\in\mathbb{F}_2$. We do so for polyhedral uncertainty.

We first consider the case where neither $\eta_1$ or $\eta_2$ are N. Let $\tilde{\mathcal{V}}_{\textrm{A}}^{\eta}$ be the relaxed uncertainty set obtained by replacing $\mathcal{W}$ with $\mathcal{W}^{\textrm{Rel},1}$ in $\breve{\mathcal{V}}_{\textrm{A}}^{\eta}$, $\eta\in\mathbb{F}\setminus \textrm{N}$. The intersection of $\tilde{\mathcal{V}}_{\textrm{A}}^{\eta_1}$ and $\tilde{\mathcal{V}}_{\textrm{A}}^{\eta_2}$, is given by (\ref{eq:interagbcrel1}).

\begin{subequations}
\label{eq:interagbcrel1}
\begin{align}
i_{\textrm{L}}^{\eta_1}=i_{\textrm{L}}^{\eta_2}:\quad&\left\{
\Gamma^{\eta_1}\left(\begin{bmatrix}
i^{\circ}_{\mathcal{C}}\\v^{\circ}_{\mathcal{S}}
\end{bmatrix} + \Sigma\lambda^{\eta_1}\right)=\Gamma^{\eta_2}\left(\begin{bmatrix}
i^{\circ}_{\mathcal{C}}\\v^{\circ}_{\mathcal{S}}
\end{bmatrix} + \Sigma\lambda^{\eta_2}\right)
\right.\label{eq:iLagiLbc}\\
\tilde{\mathcal{V}}_{\textrm{A}}^{\eta}:\quad&
\left\{\begin{array}{l}
\psi^{\eta}v_{\textrm{L}}=\Omega_z^{\eta}u_z^{\eta} + \Omega_r^{\eta}u_r^{\eta}\\
m_z^{\eta}\in[\underbar{m}_z,1],\;m_r^{\eta}\in[0,1]\\
u_z^{\eta} = m_z^{\eta}\begin{bmatrix}
i^{\circ}_{\mathcal{C}}\\v^{\circ}_{\mathcal{S}}
\end{bmatrix} + \Sigma\lambda^{\eta}\\
u_r^{\eta} = m_r^{\eta}\begin{bmatrix}
i^{\circ}_{\mathcal{C}}\\v^{\circ}_{\mathcal{S}}
\end{bmatrix} + \Sigma\lambda^{\eta}\\
P\lambda^{\eta} \leq \bm{1}
\end{array}\right.\label{eq:tildeVag}\\
&\quad\quad\quad\quad \eta\in\{\eta_1,\eta_2\}.\nonumber
\end{align}
\end{subequations}
We can tighten the relaxation by augmenting (\ref{eq:interagbcrel1}) with the constraints that make up $\mathcal{W}^{\textrm{Rel},2}$ for each $\eta\in\{\eta_1,\eta_2\}$:
\begin{align}
\begin{array}{l}
u_z^{\eta} = m_z^{\eta}\begin{bmatrix}
i^{\circ}_{\mathcal{C}}\\v^{\circ}_{\mathcal{S}}
\end{bmatrix} + \Sigma\lambda^{\eta}_{m_z}\\
u_r^{\eta} = m_r^{\eta}\begin{bmatrix}
i^{\circ}_{\mathcal{C}}\\v^{\circ}_{\mathcal{S}}
\end{bmatrix} + \Sigma\lambda^{\eta}_{m_r}\\
P\lambda^{\eta}_{m_z} \leq m_z^{\eta}\bm{1}\\
P\lambda^{\eta}_{m_r} \leq m_r^{\eta}\bm{1}.
\end{array}\label{eq:interagbcrel2}
\end{align}
(\ref{eq:interagbcrel1}) is a convex relaxation of $\breve{\mathcal{P}}^{\eta_1,\eta_2}$ based on $\mathcal{W}^{\textrm{Rel},1}$. (\ref{eq:interagbcrel1})-(\ref{eq:interagbcrel2}) is a tighter convex relaxation based on $\mathcal{W}^{\textrm{Rel},3}$.

Observe that $v_{\textrm{L}}$ is the same in (\ref{eq:tildeVag}) for both $\eta\in\{\eta_1,\eta_2\}$, and that we constrain $i_{\textrm{L}}$ to be the same in both with (\ref{eq:iLagiLbc}). We, however, do not constrain $i_{\textrm{L}}$ to be the same in both instances of (\ref{eq:interagbcrel2}). To do so, we would need to add the nonlinear constraints, which would compromise convexity.

We obtain a relaxation of $\breve{\mathcal{P}}^{\textrm{N},\eta}$, $\eta\in\mathbb{F}\setminus \textrm{N}$, by similarly joining the constraints in $\breve{\mathcal{V}}_{\textrm{A}}^{\textrm{N}}$, $\tilde{\mathcal{V}}_{\textrm{A}}^{\eta}$, and 
\[
\Gamma^{\textrm{N}}\left(\begin{bmatrix}
i^{\circ}_{\mathcal{C}}\\v^{\circ}_{\mathcal{S}}
\end{bmatrix} + \Sigma\lambda^{\textrm{N}}\right)=\Gamma^{\eta}\left(\begin{bmatrix}
i^{\circ}_{\mathcal{C}}\\v^{\circ}_{\mathcal{S}}
\end{bmatrix} + \Sigma\lambda^{\eta}\right).
\]
As before, we can tighten the relaxation by adding (\ref{eq:interagbcrel2}) for $\eta$ (but not N). 

\subsection{Post-test}

The interval uncertainty is
\begin{align}
&m_z\in[\underbar{m}_z,1],\;m_r\in[0,1].\label{eq:mzmrptapp}
\end{align}
The constraints that make up $\mathcal{U}_{m_r}$ are
\begin{subequations}
\label{eq:mzmrur}
\begin{align*}
&u_r = m_r\begin{bmatrix}
i^{\circ}_{\mathcal{C}}\\v^{\circ}_{\mathcal{S}}
\end{bmatrix} + \Sigma\lambda_{m_r}\\
&\left\|\lambda_{m_r}\right\|_2^2\leq m_r \quad\textrm{or}\quad P\lambda_{m_r}\leq m_r\bm{1},
\end{align*}
\end{subequations}
depending on if the uncertainty is ellipsoidal or polyhedral. The set $\vec{\mathcal{Z}}_{\textrm{A}}^{\eta}$ consists of (\ref{eq:mzmrptapp}), $\mathcal{U}_{m_r}$, and the constraint
\begin{align*}
z_{\textrm{A}}^{\eta} &= m_zz + m_r \xi^{\eta}  + \Xi^{\eta} \Theta^{\eta}u_r^{\eta}.
\end{align*}

\subsection{Zonotopes and zonogons}\label{app:zono}

A zonotope in $\mathbb{R}^n$ can be written
\[
\left\{\left.c+\sum_{i=1}^p\beta_ig_i\;\right|\; \beta\in[-1,1]^n\right\},
\]
where $c\in\mathbb{R}^n$ is the center and $g_i\in\mathbb{R}^n$, $i=1,...,p$, are the generators. We refer to this form as the zonotope's G-representation. The Minkowski sum of two zonotopes with centers $c_1$ and $c_2$ and generators $g_{1i}$, $i=1,...,p_1$ and $g_{2i}$, $i=1,...,p_2$, respectively, is a zonotope with center $c_1+c_2$ and both sets of generators. We refer the reader to~\cite{ziegler1995lectures} for more coverage of zonotopes.

A zonogon is a two-dimensional zonotope. We now give an algorithm with for quickly obtaining the vertices of a zonogon from its G-representation. The algorithm is based on Algorithm 3 in~\cite{girard2008zonotope}, which is related to the gift wrapping algorithm~\cite{jarvis1973identification}. Without loss of generality, we assume that
\begin{itemize}
\item all generators are pointing up,
\item there are no coincident generators,
\item and the generators are sorted by angle from nearest to zero to nearest to $\pi$.
\end{itemize}
A zonogon with $p$ generators has $2p$ vertices. Given the center and generators of a zonogon, we can efficiently obtain its vertices as follows.
\begin{enumerate}
\item Let $v_1 = c-\sum_{i=1}^pg_i$. This is the bottommost vertex.
\item For $i=2,...,p+1$, set $v_i = v_{i-1} + 2g_{i-1}$. $v_{p+1}$ is the topmost vertex.
\item For $i=2,...,p$, set $v_{p+i} = v_{p+i-1} - 2g_{i-1}$.
\end{enumerate}
The vertices of the zonogon are $v_i$, $i=1,...,2p$. This algorithm has $2p$ steps, each of which consists of vector addition.

\section{Auxiliary signal computation}\label{app:Farkas}

\subsection{Dual systems}\label{app:dualsystems}

We describe how to construct the dual systems in Lemma~\ref{lem:Farkas}, $\mathcal{S}^{\eta_1,\eta_2}(\Delta)$, $(\eta_1,\eta_2)\in\mathbb{F}_2$. One way to proceed is by writing the relaxation $\tilde{\mathcal{P}}^{\eta_1,\eta_2}(\Delta)$ in standard form; in the polyhedral case, this is $Ax=b,\; x\geq0$, and Farkas' lemma specifies that $\mathcal{S}^{\eta_1,\eta_2}(\Delta)$ is $A^{\top}y\geq0,\;b^{\top}y<0$. When optimizing, we'd typically use $b^{\top}y\leq-1$ in place of the strict inequality. Alternatively, it is often easier to derive the dual of an optimization with zero objective and feasible set $\tilde{\mathcal{P}}^{\eta_1,\eta_2}(\Delta)$, in which case $\mathcal{S}^{\eta_1,\eta_2}(\Delta)$ consists of the dual feasible set plus a constraint that the dual objective be greater than or equal to one. Complex variables and constraints can be handled by separating into real and imaginary parts, or by simply using complex multipliers and the conjugate transpose for complex constraints~\cite{craven1973duality}.

Below we list the constraints that make up $\mathcal{S}^{\eta_1,\eta_2}(\Delta)$ when the uncertainty is polyhedral and the relaxation is $\mathcal{W}^{\textrm{Rel},3}$ from Section~\ref{sec:convex:relax}. We first do so for when $\eta_1$ and $\eta_2$ are both faults, and then for normal operation, $\textrm{N}$, and a fault. We write a new variable to the left of each bilinear constraint, which we equate with the left hand side in the algorithm in Appendix~\ref{app:admm}.

When $\eta_1$ and $\eta_2$ both correspond to faults, $\mathcal{S}^{\eta_1,\eta_2}(\Delta)$ is:
\begin{align*}
\chi^{\eta_1,\eta_2}:\;&  \textrm{Re}\left[\begin{bmatrix}
i^{\circ}_{\mathcal{C}}(\Delta)\\v^{\circ}_{\mathcal{S}}
\end{bmatrix}^{\dag}\left(\Gamma^{\eta_1}-\Gamma^{\eta_2}\right)^{\dag}\alpha\right]\\
&\quad -\sum_{\eta\in\{\eta_1,\eta_2\}} \bm{1}^{\top}\phi^{\eta} + \mu_{z\textrm{u}}^{\eta}+ \mu_{r\textrm{u}}^{\eta}-\underbar{m}_z \mu_{z\textrm{l}}^{\eta}\geq1\\
& \sum_{\eta\in\{\eta_1,\eta_2\}} \beta^{\eta}\psi^{\eta}=0 \\
& \textrm{Re}\left[\Sigma^{\dag}\left(\gamma_z^{\eta_1} + \gamma_r^{\eta_1}+\Gamma^{\eta_1\dag}\alpha\right)\right] +     P^{\top}\phi^{\eta_1}=0 \\
& \textrm{Re}\left[ \Sigma^{\dag}\left(\gamma_z^{\eta_2} + \gamma_r^{\eta_2}-\Gamma^{\eta_2\dag}\alpha\right)\right] +     P^{\top}\phi^{\eta_2}=0 \\
&\hspace{-1cm}\textrm{for }\eta\in\{\eta_1,\eta_2\}:\\
\chi_{z,\eta}^{\eta_1,\eta_2}:\;& \textrm{Re}\left[\begin{bmatrix}
i^{\circ}_{\mathcal{C}}(\Delta)\\v^{\circ}_{\mathcal{S}}
\end{bmatrix}^{\dag}\left(\gamma_z^{\eta}+\sigma_z^{\eta}\right)\right] -\bm{1}^{\top}\tau_z^{\eta}+\mu_{z\textrm{u}}^{\eta} -\mu_{z\textrm{l}}^{\eta}  =0  \\
\chi_{r,\eta}^{\eta_1,\eta_2}:\;&\textrm{Re}\left[\begin{bmatrix}
i^{\circ}_{\mathcal{C}}(\Delta)\\v^{\circ}_{\mathcal{S}}
\end{bmatrix}^{\dag}\left(\gamma_r^{\eta}+\sigma_r^{\eta}\right)\right] - \bm{1}^{\top}\tau_r^{\eta}+ \mu_{r\textrm{u}}^{\eta} -\mu_{r\textrm{l}}^{\eta}  =0 \\
& \beta^{\eta}\Omega_z^{\eta\dag} - \gamma_z^{\eta} - \sigma_z^{\eta}=0 \\
& \beta^{\eta}\Omega_r^{\eta\dag} - \gamma_r^{\eta} - \sigma_r^{\eta}=0 \\
&  \textrm{Re}\left[ \Sigma^{\dag}\sigma_z^{\eta}\right] + P^{\top}\tau_z^{\eta}  = 0\\
& \textrm{Re}\left[ \Sigma^{\dag}\sigma_r^{\eta}\right] + P^{\top}\tau_r^{\eta} = 0\\
&\phi^{\eta},\tau_z^{\eta},\tau_r^{\eta},\mu_{z\textrm{l}}^{\eta}, \mu_{z\textrm{u}}^{\eta}, \mu_{r\textrm{l}}^{\eta}, \mu_{r\textrm{u}}^{\eta}\geq0
\end{align*}

For $\eta\in\mathbb{F}\setminus{\textrm{N}}$, $\mathcal{S}^{\textrm{N},\eta}(\Delta)$ is:
\begin{align*}
\chi^{\textrm{N},\eta}:\;&\textrm{Re}\left[\begin{bmatrix}
i^{\circ}_{\mathcal{C}}(\Delta)\\v^{\circ}_{\mathcal{S}}
\end{bmatrix}^{\dag}\left(\left(\Gamma^{\textrm{N}}-\Gamma^{\eta}\right)^{\dag}\alpha+\zeta^{\textrm{N}}\right)\right]\\
&\quad- \bm{1}^{\top}\left(\phi^{\textrm{N}} + \phi^{\eta}\right) - \mu_{z\textrm{u}}^{\eta}- \mu_{r\textrm{u}}^{\eta} +\underbar{m}_z \mu_{z\textrm{l}}^{\eta} \geq1\\
& \epsilon + \beta^{\eta}\psi^{\eta\top}=0\\
& \Omega^{\textrm{N}\dag}\epsilon - \zeta^{\textrm{N}} = 0 \\
& \textrm{Re}\left[ \Sigma^{\dag}\left(\Gamma^{\textrm{N}\dag}\alpha + \zeta^{\textrm{N}} \right)\right] +     P^{\top}\phi^{\textrm{N}}=0 \\
& \textrm{Re}\left[ \Sigma^{\dag}\left(\gamma_z^{\eta} + \gamma_r^{\eta} - \Gamma^{\eta\dag}\alpha\right)\right] +     P^{\top}\phi^{\eta}=0 \\
\chi_{z}^{\textrm{N},\eta}:\;& \textrm{Re}\left[\begin{bmatrix}
i^{\circ}_{\mathcal{C}}(\Delta)\\v^{\circ}_{\mathcal{S}}
\end{bmatrix}^{\dag}\left(\gamma_z^{\eta}+\sigma_z^{\eta}\right)\right]  -\bm{1}^{\top}\tau_z^{\eta}+\mu_{z\textrm{u}}^{\eta} -\mu_{z\textrm{l}}^{\eta}  =0  \\
\chi_{r}^{\textrm{N},\eta}:\;& \textrm{Re}\left[\begin{bmatrix}
i^{\circ}_{\mathcal{C}}(\Delta)\\v^{\circ}_{\mathcal{S}}
\end{bmatrix}^{\dag}\left(\gamma_r^{\eta}+\sigma_r^{\eta}\right)\right]  -\bm{1}^{\top}\tau_r^{\eta}+\mu_{r\textrm{u}}^{\eta} -\mu_{r\textrm{l}}^{\eta} =0 \\
& \beta^{\eta}\Omega_z^{\eta\dag} - \gamma_z^{\eta} - \sigma_z^{\eta}=0 \\
& \beta^{\eta}\Omega_r^{\eta\dag} - \gamma_r^{\eta} - \sigma_r^{\eta}=0 \\
&  \textrm{Re}\left[ \Sigma^{\dag}\sigma_z^{\eta}\right] + P^{\top}\tau_z^{\eta}  = 0\\
& \textrm{Re}\left[ \Sigma^{\dag}\sigma_r^{\eta}\right] + P^{\top}\tau_r^{\eta} = 0\\
& \phi^{\textrm{N}},\phi^{\eta},\tau_z^{\eta},\tau_r^{\eta},\mu_{z\textrm{l}}^{\eta}, \mu_{z\textrm{u}}^{\eta}, \mu_{r\textrm{l}}^{\eta}, \mu_{r\textrm{u}}^{\eta}\geq0
\end{align*}

\subsection{Alternating Direction Method of Multipliers}\label{app:admm}

We use the ADMM to solve (\ref{opt:AS2}), a bilinear or biconvex program. We sketch our implementation here, and refer the reader to~\cite{boyd2011distributed} for a thorough presentation of the ADMM.

We split the variables into three groups: $\Delta$, the auxiliary signal vector; $\Upsilon_1$, a vector of all variables that multiply $\Delta$; and $\Upsilon_2$, a vector of all remaining variables that make up $\mathcal{S}^{\eta_1,\eta_2}(\Delta)$, $(\eta_1,\eta_2)\in\mathbb{F}_2$. We alternate between optimizing over $\Delta$ and $\Upsilon_1$, and leave $\Upsilon_2$ variable in both optimizations. The residual vector, $\mathcal{R}$, is comprised of $\chi^{\eta_1,\eta_2}$,  $\chi_{z,\eta}^{\eta_1,\eta_2}$ and $\chi_{r,\eta}^{\eta_1,\eta_2}$, $(\eta_1,\eta_2)\in\mathbb{F}_2$, $\eta\in(\eta_1,\eta_2)$. Let $\Pi$ be a vector of the same dimension as $\mathcal{R}$. The scaled augmented Lagrangian is
\[
\mathcal{L}(\Delta,\Upsilon_1,\Upsilon_2,\Pi) = \Delta^{\dag}Q\Delta+ \rho\left\|
\mathcal{R} + \Pi
\right\|_2^2,
\]
where $\rho$ is the penalty parameter. Denote the current iteration with superscript $\alpha$. The ADMM routine is as follows.
\begin{enumerate}
\item \textit{Initialization.} Set $\alpha=0$, $\Pi^0=\bm{0}$, and $\Delta^0=\Delta_0$.
\item \textit{Minimization over $\Upsilon_1$.} Solve
\begin{align*}
\min_{\Upsilon_1,\Upsilon_2} \quad& \mathcal{L}\left(\Delta^{\alpha},\Upsilon_1,\Upsilon_2,\Pi^{\alpha}\right)\\
\textrm{such that}\quad& \mathcal{S}^{\eta_1,\eta_2}\left(\Delta^{\alpha}\right)\neq\emptyset,\quad (\eta_1,\eta_2)\in\mathbb{F}_2.
\end{align*}
Set $\Upsilon_1^{\alpha+1}$ to the optimal value of $\Upsilon_1$.
\item \textit{Minimization over $\Delta$.} Solve
\begin{align*}
\min_{\Delta,\Upsilon_2} \quad& \mathcal{L}\left(\Delta,\Upsilon_1^{\alpha+1},\Upsilon_2,\Pi^{\alpha}\right)\\
\textrm{such that}\quad& \mathcal{S}^{\eta_1,\eta_2}(\Delta)\neq\emptyset,\quad (\eta_1,\eta_2)\in\mathbb{F}_2.
\end{align*}
Set $\Delta^{\alpha+1}$ to the optimal value of $\Delta$.
\item \textit{Dual update.} Set
\[
\Pi^{\alpha+1} = \Pi^{\alpha} + \mathcal{R}^{\alpha+1}.
\]
\item \textit{Termination check.} If a termination criterion is not satisfied, set $\alpha\leftarrow\alpha+1$ and go to Step 2.
\end{enumerate}

Potential termination criteria include convergence of the objective, $\Delta^{\alpha\dag}Q\Delta^{\alpha}$, to a fixed value; convergence of the residual, $\mathcal{R}^{\alpha}$, to $\bm{0}$; and convergence of both the objective and residual.

Observe that if the initial auxiliary signal is $\Delta^0=\bm{0}$ and the residual is zero after Step 2 in the first iteration, then the models are already separated, and no auxiliary signal is needed.

\end{document}